\documentclass[a4paper,11pt,reqno]{amsart}
\usepackage{enumerate}
\usepackage{amsmath}
\usepackage{amssymb}
\usepackage{enumitem} 
\usepackage{array}
\usepackage{amsthm}
\usepackage[english]{babel}
\usepackage{youngtab}

\theoremstyle{plain}
\newtheorem{theorem}{Theorem}[section]
\newtheorem{lemma}[theorem]{Lemma}
\newtheorem{proposition}[theorem]{Proposition}
\newtheorem{corollary}[theorem]{Corollary}

\theoremstyle{definition}
\newtheorem{definition}[theorem]{Definition}
\newtheorem{example}[theorem]{Example}

\theoremstyle{remark}
\newtheorem{remark}[theorem]{Remark}

\newcommand{\Ind}{\big\uparrow}

\newcommand{\piso}{\overset{p}{\cong}}
\newcommand{\SL}{\textnormal{SL}_2(\mathbb{F}_p)}

\newcommand{\U}{\mathcal{U}}
\newcommand{\Z}{\mathbb{Z}}
\newcommand{\z}{\zeta_p}

\newcommand{\mone}{-1}
\newcommand{\mtwo}{-2}

\DeclareMathOperator{\Sym}{Sym}

\newcommand{\List}[2]{#1_1, #1_2, \dots,#1_{#2}}

\title[Representation ring and modular plethysms of $\SL$]{The representation ring of $\SL$ and stable modular plethysms of its natural module in characteristic $p$}
\author{Pavel Turek}
\date{\today}
\address{Department of Mathematics, Royal Holloway, University of London, Egham, Surrey TW20 0EX, UK}
\email{pkah149@live.rhul.ac.uk}
\begin{document}	
	\begin{abstract}
		Let $p$ be an odd prime and let $k$ be a field of characteristic $p$. We provide a practical algebraic description of the representation ring of $k\SL$ modulo projectives. We then investigate a family of modular plethysms of the natural $k\SL$-module $E$ of the form $\nabla^{\nu} \Sym^l E$ for a partition $\nu$ of size less than $p$ and $0\leq l\leq p-2$. Within this family we classify both the modular plethysms of $E$ which are projective and the modular plethysms of $E$ which have only one non-projective indecomposable summand which is moreover irreducible. We generalise these results to similar classifications where modular plethysms of $E$ are replaced by $k\SL$-modules of the form $\nabla^{\nu} V$, where $V$ is a non-projective indecomposable $k\SL$-module and $|\nu|<p$.  
	\end{abstract}

\maketitle
	
	\section{Introduction}
	
	Throughout the paper unless specified otherwise by a module we mean a finite-dimensional right module. Let $p$ be an odd prime. We fix a field $k$ of characteristic $p$. We denote by $\mathbb{F}_p$ the field with $p$ elements. Let $E$ be the natural two-dimensional module of the group $G=\SL$ over $k$. The irreducible $kG$-modules, up to isomorphism, are the symmetric powers $\Sym^0 E, \Sym^1 E, \dots, \Sym^{p-1} E$. Of these $p$ modules only $\Sym^{p-1} E$ is projective. The group ring $kG$ is of finite type, meaning it has only finitely many isomorphism classes of indecomposable modules.
	
	In this paper we study $kG$-modules modulo projectives, that is we work in the stable module category of $kG$. This is a feature which distinguishes our work from other similar publications studying $kG$-modules such as Kouwenhoven \cite{KouwenhovenLambda90a} or Hughes and Kemper \cite{HughesKemperSylow01}. An advantage of working in the stable module category of $kG$ is that understanding tensor products and Schur functors is relatively easy. Moreover, one may use results from the stable module category of $kG$ as starting points for establishing analogous results in the category of all $kG$-modules, which may be difficult to approach directly.
	
	Our first result is a novel description of the representation ring of $kG$ modulo projectives. It can be viewed as a `lift' of Almkvist's description of the representation ring of $k[\Z/p\Z]$ modulo projectives in \cite[Proposition~3.2]{AlmkvistReciprocity81} to $G$. A description of the representation ring of $kG$ (which includes the projective $kG$-modules) is established by Kouwenhoven in \cite[Corollary~1.4.2(c)]{KouwenhovenLambda90a}. As a consequence of working modulo projectives our result is more practical to work with, as can be seen in Example~\ref{height and position example}.
	
	To state our result let us denote by $\overline{R(G)}$ the representation ring of $kG$ modulo projectives and for any $kG$-module $V$ let $\overline{V}$ be the element of $\overline{R(G)}$ corresponding to $V$. Finally, write $\U_l$ for $\overline{\Sym^l E}$, $\Omega$ for the Heller operator and $k$ for the trivial $kG$-module. Then we obtain the following isomorphism involving a primitive $p$th root of unity denoted by $\z$.    
	
	\begin{theorem}\label{rep ring thm}
		The map $\Psi\!: \mathbb{Z}[\zeta_p + \zeta_p^{-1}][X,Y]/(X^{p-1}-1, Y^2-1) \to \overline{R(G)}$ given by $\zeta_p^2 + \zeta_p^{-2} \mapsto \U_2 - \U_0$, $X\mapsto \overline{\Omega k}$ and $Y\mapsto \U_{p-2}$ is an isomorphism.
	\end{theorem}
	
	Our main results concern Schur functors applied to the indecomposable $kG$-modules and modular plethysms (that is compositions of two Schur functors) of the natural $kG$-module $E$. The study of modular plethysms is related to the study of plethysms of Schur functions via formal characters. Given partitions $\lambda$ and $\mu$ the Schur functors $\nabla^{\lambda}$ and $\nabla^{\mu}$ have formal characters equal to the Schur functions $s_{\lambda}$ and $s_{\mu}$, respectively, and the modular plethysm $\nabla^{\lambda} \nabla^{\mu}$ has formal character equal to the plethysm $s_{\lambda}\circ s_{\mu}$. Finding a combinatorial description of the decomposition of the plethysm $s_{(m)}\circ s_{(n)}$ into a sum of Schur functions has been identified as one of the major open problems in algebraic combinatorics by Stanley \cite[Problem~9]{StanleyPositivity00}.
	
	In this paper our focus is on modular plethysms $\nabla^{\nu} \Sym^l E$ with $|\nu|<p$ (we refer to such partitions $\nu$ as \textit{$p$-small}) and $0\leq l\leq p-2$. Equivalently, one may characterise such modular plethysms as $\nabla^{\nu} V$ with $\nu$ a $p$-small partition and $V$ a non-projective irreducible $kG$-module.
	
	\subsection{Main results}\label{main result sec}
	
	Our first main result is the classification of the projective modular plethysms (subject to the earlier constraints). In the statement (and the rest of the paper) we write $\ell(\nu)$ for the length of partition $\nu$.
	
	\begin{theorem}\label{projective classification thm}
		Let $0\leq l\leq p-2$ and let $\nu$ be a $p$-small partition. Then $\nabla^{\nu} \Sym^l E$ is projective if and only if $\nu_1 \geq p-l$ or $\ell(\nu)\geq l+2$.
	\end{theorem}
	
	\begin{remark}
		In the case $\ell(\nu)\geq l+2$, the module $\nabla^{\nu} \Sym^l E$ equals the zero module.
	\end{remark}
	
	Next, we introduce a notion of irreducibility in the stable module category.
	
	\begin{definition}
		We say that a module is \textit{stably-irreducible} if it has precisely one non-projective indecomposable summand and this summand is irreducible.
	\end{definition} 
	
	Our second main result is the classification of modular plethysms which are stably-irreducible. It is inspired by the classification of all irreducible modular plethysms of the natural module of $\mathbb{C}$SL$_2(\mathbb{C})$ established by Paget and Wildon in \cite[Corollary~1.8]{PagetWildonPlethysms21}.
	
	\begin{theorem}\label{full irred modular plethysms of E thm}
		Let $0\leq l\leq p-2$ and let $\nu$ be a $p$-small partition. Then $\nabla^{\nu} \Sym^l E$ is stably-irreducible if and only if (at least) one of the following happens:
		\begin{enumerate}[label=\textnormal{(\roman*)}]
			\item \textnormal{(elementary cases augmented by rows)} $\nu=((p-l-1)^b,1)$ or $\nu=((p-l-1)^b)$ for some $b\geq 0$,
			\item \textnormal{(elementary cases augmented by columns)} $\nu=(a+1, a^l)$ or $\nu=(a^{l+1})$ for some $a\geq 0$,
			\item \textnormal{(augmented row cases)} $\nu = ((p-l-1)^b, p-l-2)$ for some $b\geq 0$ or $l=1$ and $\nu$ lies inside the box $2\times (p-2)$,
			\item \textnormal{(augmented column cases)} $\nu = ((a+1)^l, a)$ for some $a\geq 0$ or $l=p-3$ and $\nu$ lies inside the box $(p-2)\times 2$,
			\item \textnormal{(hook case)} $\nu = (p-l-1, 1^l)$,
			\item \textnormal{(rectangular cases)} $p=7$ and either $\nu=(2,2,2)$ with $l=3$ or $\nu=(3,3)$ with $l=2$. 
		\end{enumerate}
	\end{theorem}
	
	With additional notation a more unified classification is possible (see Theorem~\ref{irred modular plethysms of E thm}).
	
	The third and fourth main results generalise Theorem~\ref{projective classification thm} and Theorem~\ref{full irred modular plethysms of E thm}, respectively, by replacing a non-projective irreducible module $\Sym^l E$ by an arbitrary non-projective indecomposable $kG$-module.
	
	In both statements we use that any non-projective indecomposable $kG$-module is isomorphic to $\Omega^i (\Sym^l E)$ for some $0\leq i\leq p-2$ and $0\leq l\leq p-2$. This is established in Corollary~\ref{kG-Omega cor}(ii).
	
	\begin{theorem}\label{final projective classification thm}
		Let $V$ be a non-projective indecomposable $kG$-module and let $\nu$ be a $p$-small partition. Write $V\cong\Omega^i (\Sym^l E)$ for some $0\leq i\leq p-2$ and $0\leq l\leq p-2$. Then the module $\nabla^{\nu} V$ is projective if and only if $\lambda_1 \geq p-l$ or $\ell(\lambda)\geq l+2$, where $\lambda = \nu$ if $i$ is even and $\lambda = \nu'$ if $i$ is odd.
	\end{theorem}
	
	In the statement of our fourth and final result we say a pair $(\nu,l)$ of a $p$-small partition $\nu$ and $l$ with $0\leq l\leq p-2$ is \textit{stably-irreducible} if $\nabla^{\nu} \Sym^l E$ is stably-irreducible. Thus the stably-irreducible pairs are classified by Theorem~\ref{full irred modular plethysms of E thm}.
	
	\begin{theorem}\label{final classification thm}
		Let $V$ be a non-projective indecomposable $kG$-module and let $\nu$ be a $p$-small partition. Write $V\cong\Omega^i (\Sym^l E)$ for some $0\leq i\leq p-2$ and $0\leq l\leq p-2$. Then the module $\nabla^{\nu} V$ is stably-irreducible if and only if $p-1 \mid i|\nu|$ and $(\lambda, l)$ is a stably-irreducible pair, where $\lambda = \nu$ if $i$ is even and $\lambda = \nu'$ if $i$ is odd.
	\end{theorem}
	
	The key step in proving Theorem~\ref{projective classification thm} and Theorem~\ref{full irred modular plethysms of E thm} is to translate questions about projectivity and stable-irreducibility to questions about the multiset of the hook lengths of $\nu$ and the multiset of the shifted contents of $\nu$ (as introduced in \S\ref{SHCF sec}). This cannot be done by just working with polynomials as in \cite{PagetWildonPlethysms21} but one can instead use results about the group of cyclotomic units of $\Z[\z + \z^{-1}]$.
	
	The final two main results follow from the first two using a result about exchanging Schur functors and the Heller operator $\Omega$ (Corollary~\ref{Heller endo cor}). This result is of independent interest.
	
	\subsection{Notation}\label{notation sec}
	
	We use the notation $R(G)$ for the representation ring of $kG$ and $\overline{R(G)}$ for the quotient ring of $R(G)$ by the ideal generated by the isomorphism classes of the projective $kG$-modules. For each $kG$-module $V$ let $\overline{V}\in \overline{R(G)}$ denote the image of the quotient map applied to the isomorphism class of $V$ in $R(G)$. For $0\leq l\leq p-1$ we write $\U_l$ to denote $\overline{\Sym^l E}$. Hence $\U_0$ is the multiplicative identity of $\overline{R(G)}$ and $\U_{p-1}=0$. Note that we index $\U_l$ by the corresponding exponent of the symmetric power not the dimension.
	
	The ring $\overline{R(G)}$ is freely generated as an abelian group by elements of the form $\overline{V}$ where $V$ ranges over the non-projective indecomposable $kG$-modules, up to isomorphism. We define abelian subgroups $R_I$ and $R_E$ of $\overline{R(G)}$ as the subgroups generated freely by $\U_0, \U_1, \dots, \U_{p-2}$ and by $\U_0, \U_2, \dots, \U_{p-3}$, respectively. We see in Corollary~\ref{CG rule cor}(iii)~and~(v) that $R_I$ and $R_E$ are in fact subrings of $\overline{R(G)}$.
	
	We introduce the following notation to help us distinguish between isomorphisms of modules and isomorphisms of modules modulo projectives. We use $\cong$ for usual isomorphisms of modules, while $\piso$ is used for isomorphisms in the stable module category. Thus if $U$ and $V$ are two $kG$-modules, we write $U\piso V$ if there are projective $kG$-modules $P$ and $Q$ such that $U\oplus P\cong V\oplus Q$. Clearly $U\piso V$ if and only if $\overline{U}=\overline{V}$.

	\subsection{Background and results in the literature}\label{intro sub}
	
	For a detailed introduction to partitions, Schur functors, Schur functions and formal characters see \cite[\S\S2-3]{deBoeckPagetWildonPlethysms21}.
	
	Schur functors can be defined in various ways, see \cite[Definition~1.16.1]{BensonBundles17} or \cite[Definition~2.3]{deBoeckPagetWildonPlethysms21}. They can be studied using formal characters defined, for instance, in \cite[\S3.4]{GreenPolynomial80} and in turn using Schur functions defined, for instance, in \cite[Definition~7.10.1]{StanleyEnumerativeII99}. Working over the field $k$, this is particularly useful when we restrict to Schur functors $\nabla^{\nu}$ with $\nu$ a $p$-small partition since the semisimplicity result \cite[(2.6e)]{GreenPolynomial80} applies (even when $k$ is a finite field). In such a case identities involving Schur functions yield analogous identities of Schur functors (for instance, see Lemma~\ref{Schur functor expansion lemma}). This is the reason why we restrict ourselves to $p$-small partitions. 
	
	For a description of the non-projective indecomposable $kG$-modules using symmetric powers see Glover \cite[(3.3) and (3.9)]{GloverRepresentations78}. There it is established that in the stable module category of $kG$ the non-projective indecomposable $kG$-modules correspond, up to isomorphism, to the symmetric powers $\Sym^l E$ with $0\leq l< p(p-1)$ and $l$ not congruent to $-1$ modulo $p$. When it comes to the projective indecomposable $kG$-modules, their Loewy series are given, for instance, in \cite[pp.~48-52]{AlperinLocal86}. 
	
	The work of Almkvist \cite{AlmkvistReciprocity81}, \cite{AlmkvistComponents78} and \cite{AlmkvistFossumCyclic78} (revisited by Hughes and Kemper in \cite{HughesKemperCyclic00}) and results of Benson \cite[Ch.~2]{BensonBundles17} regarding the representation theory of $k[\Z/p\Z]$ can be used to study modular plethysms of $E$. One can recover most of our preliminary results using restrictions of representations and arguments used by Kouwenhoven. However, here we present a self-contained approach which does not rely on any deep results regarding the representation theory of $k[\Z/p\Z]$.    
	
	It is worth mentioning that in the literature one can find results regarding modular plethysms of $E$ which do not restrict themselves to $\nu$ being $p$-small, $0\leq l\leq p-2$ or working modulo projectives. For instance, Kouwenhoven \cite{KouwenhovenLambda90b} examines all symmetric and exterior powers of the indecomposable $kG$-modules. In a paper by Hughes and Kemper \cite{HughesKemperSylow01} the authors describe the Hilbert series of the polynomial invariants of any fixed $kG$-module (in fact, their results hold for any group $H$ with a Sylow $p$-subgroup of order $p$ in place of $G$). 
	
	A slightly different task is accomplished by McDowell and Wildon in \cite{McDowellWildonIsomorphisms22}. They examine whether there exist modular versions of various families of isomorphisms of modular plethysms of the natural $\mathbb{C} $SL$_2(\mathbb{C})$-module (such as the Hermite reciprocity or the Wronskian isomorphism).
	
	We should also note that the use of a $p$th root of unity to describe projective objects of a group ring (in our case modular plethysms of $E$ in Theorem~\ref{projective classification thm}) has already appeared in the literature. In particular, Theorem of Rim \cite[p.~29]{MilnorKTheory71} is a result in algebraic $K$-theory which states that there is an isomorphism between $K_0\Z\Pi$ and $K_0\Z[\z]$ where $\Pi = \Z/p\Z$.
	
	While many different approaches to studying the representation theory of $kG$ appear in the literature, let us mention that our results about endotrivial modules in \S\ref{endo sec} provide a unification to some of them. In particular, several arguments of Kouwenhoven \cite{KouwenhovenLambda90a}, \cite{KouwenhovenLambda90b} and Glover \cite{GloverRepresentations78} can be deduced from our results and the facts that the module $\Sym^p E$ is endotrivial and for any non-negative integer $a$ we have $\Sym^a E\otimes \Sym^p E \piso\Sym^{a+p} E$ (see \cite[Corollary~1.4.1(b) and Corollary~1.5(a)]{KouwenhovenLambda90a}).  
	
	\subsection{Outline}
	
	The paper is divided into five sections, the first of which is the introduction.
	
	In the second section we establish preliminary results. We start by recalling the Clebsch--Gordan rule and results obtained using the Green correspondence and the Heller operator $\Omega$. The rest of the section focuses on results regarding Schur functors and Schur functions. We recall Stanley's Hook Content Formula, Schur--Weyl duality and basics from the theory of $\lambda$-rings and $p$-$\lambda$-rings. 
	
	The goal of the third section is to describe how to interchange Schur functors and the operation of tensoring with an endotrivial module (in particular, with $\Omega k$).
	
	In the fourth section we prove Theorem~\ref{rep ring thm}, establishing the structure of the representation ring of $kG$ modulo projectives. We end by providing a diagrammatic interpretation of Theorem~\ref{rep ring thm}.  
	
	The goal of the final section is to prove our four main results. We start by providing a description of modular plethysms of $E$ using a homomorphism $\Theta$ from the fourth section. Theorem~\ref{projective classification thm} and Theorem~\ref{final projective classification thm} are then easily deduced. After introducing the cyclotomic units of the ring $\Z[\z + \z^{-1}]$ we prove the remaining two main results.
	
	\section{Preliminaries}
	
	Recall we use $G$ to denote the group $\SL$ and $E$ to denote the natural two-dimensional $kG$-module. We have already mentioned the following result about the irreducible $kG$-modules.
	
	\begin{theorem}\label{irreducible kG modules thm}
		The irreducible $kG$-modules, up to isomorphism, are given by $\Sym^0 E, \Sym^1 E, \dots, \Sym^{p-1} E$. Of these $kG$-modules only $\Sym^{p-1} E$ is projective.
	\end{theorem}
	
	\begin{proof}
		See \cite[pp.~14-16]{AlperinLocal86} for the first statement and \cite[p.~79]{AlperinLocal86} for the final statement.
	\end{proof}

	Note that for any $0\leq l\leq p-1$ the dimension of $\Sym^l E$ is $l+1$ and $\Sym^0 E$ is the trivial $kG$-module which we may denote by $k$.
	\subsection{Clebsch--Gordan rule}
	
	The decomposition of the tensor products of the form $\Sym^i E\otimes \Sym^j E$ with $0\leq i,j\leq p-1$ into indecomposable summands, known as a Clebsch--Gordan rule, has been established, for instance, in \cite[Theorem~3.7]{McDOwellWalk22}, \cite[(5.5) and (6.3)]{GloverRepresentations78} and \cite[Corollary~1.2(a) and Proposition~1.3(c)]{KouwenhovenLambda90a}. We will use the following version set in the stable module category of $kG$.
	
	\begin{theorem}[Stable Clebsch--Gordan rule]\label{CG rule thm}
		Let $0\leq i\leq j\leq p-2$. Writing $S$ for the tensor product $\Sym^i E\otimes \Sym^j E$ we have the decomposition
		\begin{align*}
		S \piso
		\begin{cases}
		\Sym^{j-i} E \oplus \Sym^{j-i+2} E \oplus \dots \oplus \Sym^{i+j} E &\text{if } i+j < p-2,\\
		\Sym^{j-i} E \oplus \Sym^{j-i+2} E \oplus \dots \oplus \Sym^{2p-4-i-j} E &\text{if } i+j \geq p-2.\\
		\end{cases}
		\end{align*}
	\end{theorem}
	
	\begin{example}\label{Tensoring with E example}
		By taking $i=1$ and $1\leq j\leq p-3$ in Theorem~\ref{CG rule thm}, we compute $E\otimes \Sym^j E\piso \Sym^{j-1} E \oplus \Sym^{j+1} E$. In fact, since both sides have dimension $2j+2$ and the right-hand side has no projective indecomposable summands, we have $E\otimes \Sym^j E\cong \Sym^{j-1} E \oplus \Sym^{j+1} E$.
	\end{example}
	
	We can immediately make a few simple observations using our notation from \S\ref{notation sec}.
	
	\begin{corollary}\label{CG rule cor}
		Suppose that $0\leq i\leq j\leq p-2$ and $0\leq l\leq p-2$.
		\begin{enumerate}[label=\textnormal{(\roman*)}]
			\item \begin{align*}
			\U_i \cdot \U_j =
			\begin{cases}
			\U_{j-i} + \U_{j-i+2} + \dots + \U_{i+j} &\text{if } i+j < p-2,\\
			\U_{j-i} + \U_{j-i+2} + \dots + \U_{2p-4-i-j} &\text{if } i+j \geq p-2.\\
			\end{cases}
			\end{align*}
			
			\item $\U_l \cdot \U_{p-2} = \U_{p-2-l}$.
			
			\item The abelian group $R_I$ is a subring of $\overline{R(G)}$.
			
			\item The ring $R_I$ is generated by $\U_1$ as a $\Z$-algebra.
			
			\item The abelian group $R_E$ is a subring of $\overline{R(G)}$ and $R_I$. 
		\end{enumerate}
	\end{corollary}
	
	\begin{proof}\leavevmode
		\begin{enumerate}[label=\textnormal{(\roman*)}]
			\item This is just Theorem~\ref{CG rule thm} with a different notation.
			\item Take $i=l$ and $j=p-2$ (so $i\leq j$) in (i). Then $i+j\geq p-2$ and we are in the second case. Since $2j=2p-4$ we have $j-i = 2p-4-i-j$, and hence we get only one summand which is $\U_{p-2-l}$.
			\item From (i) we see that $R_I$ is closed under multiplication.
			\item Suppose that the subring $\Z[\U_1]$ is a proper subring of $R_I$, and thus there is the least $0\leq m\leq p-2$ such that $\U_m$ does not lie in $\Z[\U_1]$. Certainly, $m\geq 2$ and by Example~\ref{Tensoring with E example} with $j=m-1$ (lying in the correct range) we have that $\U_m=\U_1 \cdot \U_{m-1} - \U_{m-2}$. By the minimality of $m$, the right-hand side lies in $\Z[\U_1]$ and hence so does $\U_m$, a contradiction.
			\item From (i) we see that $R_E$ is closed under multiplication since if both $i$ and $j$ are even, so is $j-i$ and hence so are the indices of all the summands on the right-hand side in the formula in (i).
			\qedhere
		\end{enumerate}
	\end{proof}
	
	\begin{example}\label{Ideal example}
		Let us use Corollary~\ref{CG rule cor}(ii) to describe the ideal of $R_I$ generated by $\U_0 + \U_{p-2}$. We compute that for any $0\leq j\leq p-2$ the product $(\U_0+\U_{p-2})\cdot \U_j$ equals $\U_j + \U_{p-2-j}$. Since $p$ is odd, the ideal considered as an abelian group is freely generated by the elements $\U_{2i} + \U_{p-2-2i}$ with $0\leq i \leq (p-3)/2$.
	\end{example}
	
	\subsection{Green correspondence and the indecomposable $kG$-modules}\label{Green sec}
	
	Using the Green correspondence to study the representation theory of $kG$ is common in the literature (for instance, see \cite[pp.~75-79]{AlperinLocal86} or \cite[pp.~434-438]{GloverRepresentations78}), thus we only summarise the notation and the basic results we need.
	
	Let $c$ be a generator of the cyclic group $\left( \mathbb{F}_p\right)  ^{\times}$. Consider the following two elements of $G$:
	\[g=
	\begin{pmatrix}
	1 & 0\\
	1 & 1\\
	\end{pmatrix}
	\text{ and } h=
	\begin{pmatrix}
	c^{-1} & 0\\
	0 & c\\
	\end{pmatrix}.
	\]
	
	Then $g$ generates a group of order $p$, call it $P$. The group $P$ is a Sylow $p$-subgroup of $G$ and its normaliser, which we refer to as $N$, is generated by $g$ and $h$. For each $i\in \Z/(p-1)\Z$ we define a one-dimensional irreducible $kN$-module $S_i$ by letting $g$ act trivially and $h$ act by a multiplication by $c^i$. We can describe the indecomposable $kN$-modules using the following proposition.
	
	\begin{proposition}\label{mod-kN prop}
		For each $i\in \Z/(p-1)\Z$ and $0\leq j\leq p-1$ there is a unique, up to isomorphism, uniserial $kN$-module with composition factors $S_i, S_{i-2}, S_{i-4}, \dots, S_{i-2j}$ from top to bottom. These $p(p-1)$ modules form all the indecomposable $kN$-modules, up to isomorphism. Moreover, an indecomposable $kN$-module with factors $S_i, S_{i-2}, S_{i-4}, \dots, S_{i-2j}$ from top to bottom is projective precisely when $j=p-1$. 
	\end{proposition}

	\begin{proof}
		See \cite[pp.~35-37,~76]{AlperinLocal86}.
	\end{proof}
	
	Write $U_{i,j}$ for a fixed indecomposable $kN$-module with composition factors $S_i, S_{i-2}, S_{i-4}, \dots, S_{i-2j}$ from top to bottom.   
	
	The next set of results we need concerns the Green correspondence in the case of our groups $G$ and $N$ and also the Heller operator $\Omega$.
	
	\begin{proposition}\label{Green prop}
		The restriction of $kG$-modules to $N$ yields a bijection between the non-projective indecomposable modules of $kG$ and $kN$ when working modulo projective modules and up to isomorphism. Under this bijection $\Sym^l E$ (where $0\leq l \leq p-2$) corresponds to $U_{l,l}$.
	\end{proposition}

	\begin{proof}
		See \cite[Theorem~10.1]{AlperinLocal86} for the first statement and \cite[p.~76]{AlperinLocal86} for the final statement.
	\end{proof}

	\begin{proposition}\label{Heller prop}
		Let $H$ denote a finite group of order divisible by $p$ and let $V$ and $W$ be two $kH$-modules.
		\begin{enumerate}[label=\textnormal{(\roman*)}]
			\item The module $V$ is projective if and only if $\Omega(V)$ is projective (in which case $\Omega(V)=0$).
			
			\item The Heller operator $\Omega$ defines a bijection on the class of the isomorphism classes of the non-projective indecomposable $kH$-modules. 
			
			\item $\Omega(V \oplus W) \cong \Omega V \oplus \Omega W$.
			
			\item The Heller operator $\Omega$ commutes with the bijection in Proposition~\ref{Green prop}.
			
			\item If $k$ denotes the trivial $kH$-module, then $(\Omega k)\otimes V\piso \Omega V$. 
		\end{enumerate}
	\end{proposition}

\begin{proof}
	See \cite[\S20]{AlperinLocal86}.
\end{proof}

\begin{remark}\label{Stable equivalences}
	In fact, the restriction in Proposition~\ref{Green prop} and the Heller operator in Proposition~\ref{Heller prop} define stable equivalences between $kG$ and $kN$, respectively, $kH$ and itself; see \cite[Theorem~10.3]{AlperinLocal86}, respectively, \cite[Lemma~20.6]{AlperinLocal86}. 
\end{remark}
	
	Applying the Heller operator $\Omega$ to the non-projective indecomposable $kN$-modules is easy to describe. 
	
	\begin{example}\label{Omega action example}
		Let $i\in \Z/(p-1)\Z$ and $0\leq j\leq p-2$. Using Proposition~\ref{mod-kN prop}, the projective cover of $U_{i,j}$ is $U_{i,p-1}$ and in turn $\Omega U_{i,j} = U_{i-2j-2,p-j-2}$. Applying $\Omega$ once more we obtain the identity $\Omega^2 U_{i,j} = U_{i-2,j}$. 
	\end{example}
	
	Our next task is to describe the orbits of the bijection given by applying $\Omega$ to the non-projective indecomposable $kN$-modules. 
	
	\begin{lemma}\label{kN-Omega lemma}
		There are $p-1$ orbits of the bijection given by applying $\Omega$ to the non-projective indecomposable $kN$-modules. Each has size $p-1$ and the non-projective indecomposable modules $U_{l,l}$ (with $0\leq l\leq p-2$) form a set of representatives of these orbits.
	\end{lemma}
	
	\begin{proof}
		Recall that by Proposition~\ref{mod-kN prop} if we let $i\in \Z/(p-1)\Z$ and $0\leq j\leq p-2$, we obtain all the non-projective indecomposable $kN$-modules, up to isomorphism, by considering $U_{i,j}$. We now fix $i,j$ in this range.
		
		From Example~\ref{Omega action example} we see that iteratively applying of $\Omega$ to $U_{i,j}$ results in an alternation of the second index between $j$ and $p-2-j$ (these are distinct as $p$ is odd). Therefore $\Omega$ must be applied evenly many times to get back to $U_{i,j}$. Since $\Omega^2 U_{i,j} = U_{i-2,j}$ we can immediately see that we need to apply $\Omega^2$ (at least) $(p-1)/2$ to get back to $U_{i,j}$, hence each orbit has size $p-1$. This also means that there are $p-1$ orbits as the total number of the non-projective indecomposable $kN$-modules, up to isomorphism, is $(p-1)^2$ by Proposition~\ref{mod-kN prop}.
		
		It remains to show that $U_{l,l}$ with $0\leq l\leq p-2$ form a set of representatives. Indeed, if $U_{l,l}$ and $U_{l',l'}$ (with $l\neq l'$) share an orbit, then by using the observed alternation of the second index in the previous paragraph we would have $l'=p-2-l$. But since applying $\Omega$ does not change the well-defined parity of the first index of $U_{i,j}$, we cannot have $U_{l,l}$ and $U_{{p-2-l,p-2-l}}$ in the same orbit. Thus all $U_{l,l}$ lie in different orbits and as there are as many of them as there are orbits they form a set of orbit representatives.     
	\end{proof}
	
	As an immediate corollary we obtain an analogous result for the non-projective indecomposable $kG$-modules and in turn a description of all such $kG$-modules.
	
	\begin{corollary}\label{kG-Omega cor}
		Write $M$ for the set of pairs $(l,m)$ with $0\leq l\leq p-2$ and $m\in \Z/(p-1)\Z$.
		\begin{enumerate}[label=\textnormal{(\roman*)}]
			\item There are $p-1$ orbits of the bijection given by applying $\Omega$ to the non-projective indecomposable $kG$-modules. Each has size $p-1$ and the non-projective irreducible modules $\Sym^l E$ (where $0\leq l\leq p-2$) form a set of representatives of these orbits.
			
			\item There is a one to one correspondence between $M$ and the non-projective indecomposable $kG$-modules, up to isomorphism, given by pairing $(l,m)$ with $\Omega^m (\Sym^l E)$.
			
			\item There is a one to one correspondence between $M$ and the free generators of the abelian group $\overline{R(G)}$ consisting of $\overline{V}$, where $V$ runs over the non-projective indecomposable $kG$-modules, up to isomorphism, given by pairing $(l,m)$ with $\overline{\Omega k}^m \cdot \U_l$.
		\end{enumerate}
	\end{corollary}
	
	\begin{proof}\leavevmode
		\begin{enumerate}[label=\textnormal{(\roman*)}]
			\item This is just a combination of Proposition~\ref{Green prop}, Proposition~\ref{Heller prop}(iv) and Lemma~\ref{kN-Omega lemma}.
			
			\item This is immediate from (i).
			
			\item This follows from (ii) after using Proposition~\ref{Heller prop}(v).
			\qedhere 
		\end{enumerate}
	\end{proof}
	
	\subsection{Partitions and Stanley's Hook Content Formula}\label{SHCF sec}
	
	Let $n$ be a non-negative integer and let $\lambda = (\lambda_1, \lambda_2, \dots, \lambda_t)$ be a partition of $n$. We denote the length of $\lambda$ by $\ell(\lambda)$, the size of $\lambda$ by $|\lambda|$ and write $\lambda \vdash n$ to denote that $\lambda$ is a partition of $n$. We also use the notation $\lambda'$ for the conjugate partition of $\lambda$. Finally, we denote the Schur function labelled by $\lambda$ by $s_{\lambda}$.
	
	For $\lambda$ a partition write $Y(\lambda) =\left\lbrace (i,j)\!: i\leq \ell(\lambda) \text{ and } j\leq \lambda_i \right\rbrace $ for its Young diagram. For each $(i,j)\in Y(\lambda)$ we denote its \textit{hook length}, equal to $\lambda_i + \lambda'_j - i -j +1$, by $h_{i,j}$. Similarly, we denote its \textit{content}, equal to $j-i$, by $c_{i,j}$. We denote the multisets of the hook lengths and contents of $\lambda$ by $\mathcal{H}(\lambda)$ and $\mathcal{C}(\lambda)$, respectively. Finally, for any $s\in \Z$ we denote by $\mathcal{C}_s(\lambda)$ the multiset of the \textit{shifted contents by $s$} of $\lambda$, which is obtained from $\mathcal{C}(\lambda)$ by adding $s$ to all the elements of $\mathcal{C}(\lambda)$. Throughout the paper we simply write $\mathcal{H}$, $\mathcal{C}$ and $\mathcal{C}_s$ for the above multisets as the underlying partition is always implicit from the context.
	
	\begin{example}\label{hook lengths and contents example}
		Let $\lambda = (4,3,1)$ and $s=3$. Figure~\ref{H and C figure} shows that $\mathcal{H}=\left\lbrace 1,1,1,2,3,4,4,6\right\rbrace, \mathcal{C}=\left\lbrace -2,-1,0,0,1,1,2,3 \right\rbrace$ and $\mathcal{C}_s=\left\lbrace 1,2,3,3,4,4,5,6 \right\rbrace$.
	\end{example}
	
	\begin{figure}[h]
		\Yboxdim19pt
		$\young(6431,421,1) \qquad  \young(0123,\mone 01,\mtwo) \qquad \young(3456,234,1)$
		\caption{Three Young diagrams of the partition $\lambda = (4,3,1)$ displaying its hook lengths, contents and shifted contents by $3$ from left to right.}
		\label{H and C figure}
	\end{figure}
	
	Let us now make a few simple observations about shifted contents.
	
	\begin{lemma}\label{shifted content lemma}
		Let $\lambda$ be a non-empty partition and $s$ an integer.
		\begin{enumerate}[label=\textnormal{(\roman*)}]
			\item The maximal element of $\mathcal{C}_s$ is $\lambda_1 + s -1$ and its multiplicity in $\mathcal{C}_s$ is one.
			\item The minimal element of $\mathcal{C}_s$ is $s +1 - \ell(\lambda)$ and its multiplicity in $\mathcal{C}_s$ is one.
			\item An integer $m$ lies in $\mathcal{C}_s$ if and only if $s +1 - \ell(\lambda)\leq m \leq \lambda_1 + s -1$.
		\end{enumerate}
	\end{lemma}
	
	\begin{proof}
		From the definition of $\mathcal{C}_s$, the maximal element of $\mathcal{C}_s$ corresponds to the rightmost box in the first row. Similarly, the minimal element corresponds to the bottom box in the first column. Hence we get (i) and (ii). The `only if' direction in (iii) comes from (i) and (ii), while the `if' direction comes from considering the shifted contents of the boxes in the first row and the first column.
	\end{proof}
	
	The following slightly modified Stanley's Hook Content Formula appears in \cite[Proposition~2.5.7]{BensonBundles17} (while the standard version can be found in \cite[Theorem~7.21.2]{StanleyEnumerativeII99}). It is an essential tool for establishing the main results.
	
	\begin{theorem}[Stanley's Hook Content Formula]\label{SHCF thm}
		Let $\lambda$ be a partition, $l$ a non-negative integer and $q$ a variable. Then
		\[
		s_{\lambda}(q^{-l}, q^{-l+2}, \dots, q^l) = \frac{\prod_{c\in \mathcal{C}_{l+1}} (q^c - q^{-c})}{\prod_{h\in \mathcal{H}} (q^h - q^{-h})}.
		\]
	\end{theorem}
	
	\subsection{Identities of Schur functors}
	
	Throughout this subsection let $H$ be a finite group. We denote the Schur functor labelled by $\lambda$ by $\nabla^{\lambda}$, the symmetric group on $n$ elements by $S_n$ and the Specht module labelled by $\lambda$ by $S^{\lambda}$.
	
	An initial result regarding Schur functors we need is Schur--Weyl duality. This name refers to two different statements which are usually presented in characteristic zero. The version stated below can be recovered from \cite[Lemma~2.4]{BryantLiePowers09} using the same strategy as in characteristic zero and the fact that the group algebra $kS_n$ is semisimple if $n<p$.
	
	\begin{theorem}[Schur--Weyl duality]\label{SW duality thm}
		Suppose that $n<p$ and $V$ is a $kH$-module. Then there is an isomorphism of $kS_n$-$kH$-bimodules
		\[
		V^{\otimes n} \cong \bigoplus_{\lambda \vdash n} S^{\lambda}\otimes \nabla^{\lambda} V. 
		\]
	\end{theorem}
	
	Note that the left action of the symmetric group $S_n$ on $V^{\otimes n}$ is given by place permutation. That is, for $\sigma \in S_n$ and $\List{x}{n}\in V$ we define $\sigma \left( x_1\otimes x_2\otimes \dots \otimes x_n\right)  = x_{1\sigma}\otimes x_{2 \sigma}\otimes \dots \otimes x_{n \sigma}$ and extend this linearly to $V^{\otimes n}$.
	
	Theorem~\ref{SW duality thm} tell us that if $V$ is a $kH$-module and $\nu$ is a partition of $n$ with $n<p$, then $\nabla^{\nu} V$ is a summand of $V^{\otimes n}$. This together with rings $R_I$ and $R_E$ introduced in \S\ref{notation sec} can be used to establish our initial results about modular plethysms of the natural $kG$-module $E$. Recall that we say that a partition $\nu$ is $p$-small if $|\nu|<p$.
	
	\begin{lemma}\label{initial modular plethysms lemma}
		Suppose that $\nu$ is a $p$-small partition and $0\leq l\leq p-2$.
		\begin{enumerate}[label=\textnormal{(\roman*)}]
			\item All the non-projective indecomposable summands of $\nabla^{\nu} \Sym^l E$ are irreducible.
			\item If $\Sym^a E$ and $\Sym^b E$ (with $0\leq a,b,\leq p-2$) are both summands of $\nabla^{\nu} \Sym^l E$, then $a$ and $b$ have the same parity.
		\end{enumerate}
	\end{lemma}
	
	\begin{proof}
		By Schur--Weyl duality $\nabla^{\nu} \Sym^l E$ is a summand of $\left( \Sym^l E\right)^{\otimes |\nu|}$. Now, we know that all the non-projective indecomposable summands of this $|\nu|$-fold tensor product are irreducible from the Clebsch--Gordan rule (in particular, using that $R_I$ is closed under multiplication from Corollary~\ref{CG rule cor}(iii)). And thus the same is true for $\nabla^{\nu} \Sym^l E$, which establishes (i).
		
		To prove (ii), recall that $\U_l = \overline{\Sym^l E}$ and that we can write $\U_l$ as $\U_{p-2}^{\varepsilon}\cdot \U_{2i}$ for some $0\leq i\leq (p-3)/2$ and $\varepsilon\in \left\lbrace 0,1\right\rbrace $, following from Corollary~\ref{CG rule cor}(ii). In the ring $\overline{R(G)}$ the $|\nu|$-fold tensor product $\left( \Sym^l E\right) ^{\otimes |\nu|}$ becomes $\U_{p-2}^{\varepsilon |\nu|}\cdot \U_{2i}^{|\nu|}$. The latter element in the product lies in $R_E$ (by Corollary~\ref{CG rule cor}(v)), thus it is a sum of certain $\U_j$ for even $j$. Hence all summands of $\U_{p-2}^{\varepsilon |\nu|}\cdot \U_{2i}^{|\nu|}$ are of the form $\U_j$ with $j$ even (if $\varepsilon |\nu|$ is even) or with $j$ odd (if $\varepsilon |\nu|$ is odd). Hence if $\Sym^a E$ and $\Sym^b E$ are both summands of $\left( \Sym^l E\right) ^{\otimes |\nu|}$, then the parities of $a$ and $b$ agree. The same must then hold for $\nabla^{\nu} \Sym^l E$, establishing (ii).
	\end{proof}
	
	We now briefly recall two famous sets of coefficients.
	
	\begin{definition}\label{L-R defn}
		Suppose that $K$ is a field of characteristic zero and $a$ and $b$ are two non-negative integers. Define $n=a+b$ and let $\lambda, \mu$ and $\nu$ be partitions of $a, b$ and $n$, respectively. The \textit{Littlewood--Richardson coefficient} $c_{\lambda\mu}^{\nu}$ equals the multiplicity of $S^{\nu}$ as a summand of $\left( S^{\lambda} \boxtimes S^{\mu} \right)\Ind_{S_a\times S_b}^{S_n}$.
	\end{definition}
	
	\begin{definition}\label{Kronecker defn}
		Suppose that $K$ is a field of characteristic zero, $n$ is a non-negative integer and $\lambda, \mu, \nu \vdash n$. The \textit{Kronecker coefficient} $g_{\lambda \mu}^{\nu}$ equals the multiplicity of the Specht module $S^{\nu}$ as a summand of $S^{\lambda}\otimes S^{\mu}$.
	\end{definition}
	
	While computing the Kronecker coefficients is a difficult task in general, we only need their values in particular cases.
	
	\begin{lemma}\label{Kronecker coefficients lemma}
		Let $n$ be a non-negative integer and $\mu, \nu\vdash n$.
		\begin{enumerate}[label=\textnormal{(\roman*)}]
			\item $g_{(n) \mu}^{\nu} = 0$ unless $\mu = \nu$, in which case the value of the Kronecker coefficient is $1$.
			\item $g_{(1^n) \mu}^{\nu} = 0$ unless $\mu = \nu'$, in which case the value of the Kronecker coefficient is $1$.
		\end{enumerate}
	\end{lemma}
	
	We have introduced the Littlewood-Richardson and the Kronecker coefficients as they appear in the following expansions.
	
	\begin{lemma}\label{Schur functor expansion lemma}
		Suppose that $n<p$ and $V,W$ are two $kH$-modules. Then for any $\nu \vdash n$ we have isomorphisms
		\[
		\nabla^{\nu} \left( V\oplus W\right)  \cong \bigoplus_{\lambda,\mu :\; |\lambda| + |\mu|=n} c_{\lambda\mu}^{\nu}  \left( \nabla^{\lambda} V \otimes \nabla^{\mu} W\right)
		\]
		and
		\[
		\nabla^{\nu} \left( V\otimes W\right)  \cong \bigoplus_{\lambda,\mu\vdash n} g_{\lambda\mu}^{\nu}  \left( \nabla^{\lambda} V \otimes \nabla^{\mu} W\right).    
		\]
	\end{lemma}
	
	\begin{proof}
		Using the semisimplicity result \cite[(2.6e)]{GreenPolynomial80}, this follows after taking formal characters and using \cite[(5.9) and (7.9)]{MacdonaldPolynomials95}.
	\end{proof}
	
	A crucial consequence of the first isomorphism is that for any $p$-small partition $\nu$ the Schur functor $\nabla^{\nu}$ is a well-defined functor in the stable module category. Indeed, we have the following result.
	
	\begin{corollary}\label{well defined functors cor}
		Suppose that $\nu$ is a $p$-small partition and $V,P$ are two $kH$-modules with $P$ being projective. Then $\nabla^{\nu}\left(V\oplus P \right) \piso \nabla^{\nu} V$.
	\end{corollary}
	
	\begin{proof}
		Let $n=|\nu| < p$. By Schur--Weyl duality if $\mu$ is a non-empty $p$-small partition, the module $\nabla^{\mu} P$ is a summand of $P^{\otimes |\mu|}$ and hence it is projective. Therefore the summands on the right-hand side of $\nabla^{\nu} \left( V\oplus P\right)  \cong \bigoplus_{\lambda,\mu :\; |\lambda| + |\mu|=n} c_{\lambda\mu}^{\nu}  \left( \nabla^{\lambda} V \otimes \nabla^{\mu} P\right)$ (given by Lemma~\ref{Schur functor expansion lemma}) are all projective apart from the ones from $\mu=\o$. Thus 
		\[\nabla^{\nu}\left(V\oplus P \right) \piso \bigoplus_{\lambda\vdash n} c_{\lambda\o}^{\nu}\nabla^{\lambda} V.\]
		From Definition~\ref{L-R defn} clearly $c_{\lambda\o}^{\nu} = 1$ if $\lambda=\nu$ and $c_{\lambda\o}^{\nu} = 0$ otherwise, establishing the result.
	\end{proof}
	
	\subsection{$p$-$\lambda$-rings}\label{p subsection}
	
	Here we only provide a brief summary about $\lambda$-rings and $p$-$\lambda$-rings needed for this paper. For more detailed background see, for instance, \cite{KnutsonLambda73} or \cite[\S2.6]{BensonBundles17} (note that in the latter the term special $\lambda$-rings is used for $\lambda$-rings and similarly, the term special $p$-$\lambda$-rings is used for $p$-$\lambda$-rings).
	
	Recall that \textit{$\lambda$-ring} is a ring equipped with operations $\lambda^i$ for all $i\in \Z_{\geq 0}$ satisfying several axioms (see \cite[pp.~7,~13]{KnutsonLambda73}). A ring is a \textit{$p$-$\lambda$-ring} if the operations $\lambda^i$ are defined for all $i<p$ and they satisfy the same set of axioms as in the case of $\lambda$-rings but only restricted to powers of $\lambda$ which are less than $p$ and in the case of the axiom regarding the composition $\lambda^i(\lambda^j(x))$ the condition $ij<p$ is required.
	
	\begin{example}\label{p-lambda first example}
		An example of a $\lambda$-ring is the ring of symmetric functions over $\Z$. If $f$ is a symmetric function over $\Z$ with all coefficients non-negative, then $\lambda^i$ applied to $f$ is the elementary symmetric polynomial $e_i$ evaluated at the monomials of $f$.
		
		For example, if $f=e_2(x_1,x_2,\dots)=x_1x_2 + x_1x_3 + x_2x_3 + \dots$, then $\lambda^3 f = e_3(x_1x_2,x_1x_3,x_2x_3,\dots)$.
	\end{example}
	
	For the next example let $\z$ denote a $p$th root of unity.
	
	\begin{example}\label{p-lambda second example}
		By specialising Example~\ref{p-lambda first example} to two variables, we get a $\lambda$-ring consisting of the symmetric functions of $\Z[x_1,x_2]$. By a further specialisation $x_1 = \z$ and $x_2 =\z^{-1}$, one obtains a $p$-$\lambda$-ring $\Z[\z +\z^{-1}]$. If $f\in\Z[\z +\z^{-1}]$ is a sum of powers of $\zeta_p$, then $\lambda^i(f)$ is given by the elementary polynomial $e_{i}$ evaluated at these powers.
		
		For example, if $f=2\z^2 + 1 + 2\z^{-2}$, then $\lambda^3 f= e_3(\z^2,\z^2,1,\z^{-2},\z^{-2}) = \z^4 + 2\z^2 + 4 + 2\z^{-2} + \z^{-4}$.
	\end{example}
	
	\begin{example}\label{p-lambda third example}
		Another example of a $p$-$\lambda$-ring is the representation ring of an arbitrary group (modulo projectives) in characteristic $p$. $\lambda^i$ is given by the $i$th exterior power and the axioms can be recovered using formal characters. This makes $\overline{R(G)}$ into a $p$-$\lambda$-ring. Using Lemma~\ref{initial modular plethysms lemma}(i), we can see that $R_I$ is closed under the operations $\lambda^i$ (with $i<p$), in other words, for any $0\leq l \leq p-2$ the module $\bigwedge^i \Sym^l E$ has all the non-projective indecomposable summands irreducible. Therefore $R_I$ is a $p$-$\lambda$-subring of $\overline{R(G)}$.
	\end{example}
	
	For any partition $\nu$ we can define an operation $\left\lbrace \nu \right\rbrace $ on $\lambda$-rings by $\left\lbrace \nu\right\rbrace x = \det(\lambda^{\nu'_i +j -i}(x))_{i,j\leq \nu_1}$. This applies to $p$-$\lambda$-rings as long as $\nu_1 + \ell(\nu)-1 < p$. In particular, if $\nu$ is $p$-small. For such $\nu$ one can use formal characters to see that the operation $\left\lbrace \nu\right\rbrace $ on the representation ring of an arbitrary group (modulo projectives) coincides with the Schur functor $\nabla^{\nu}$. In the case of the $p$-$\lambda$-ring $\Z[\z + \z^{-1}]$ from Example~\ref{p-lambda second example} the operation $\left\lbrace \nu \right\rbrace$ applied to $f\in\Z[\z +\z^{-1}]$ which is a sum of powers of $\zeta_p$ returns the Schur function $s_{\nu}$ evaluated at these powers.
	
	\begin{example}
		If $f=2\z^2 +1 +2\z^{-2}\in \Z[\z + \z^{-1}]$, then $\left\lbrace (2,1) \right\rbrace f=s_{2,1}(\z^2,\z^2,1,\z^{-2},\z^{-2})=2\z^6 + 4\z^4 + 10\z^2 + 8 +10\z^{-2} +4\z^{-4} + 2\z^{-6}$.
	\end{example}  
	
	\section{Endotrivial modules}\label{endo sec}
	
	Throughout the section let $H$ be a finite group of order divisible by $p$. Recall that a $kH$-module $V$ is called \textit{endotrivial} if $V\otimes V^* \piso k$. Note that alternatively one may replace this condition by the following equivalent version: a $kH$-module $V$ is endotrivial if and only if there is a $kH$-module $W$ such that $V\otimes W\piso k$. In particular, if there is a positive integer $n$ such that $V^{\otimes n}\piso k$, then $V$ is endotrivial. 
	
	An example of an endotrivial module is $\Omega k$ where $k$ is the trivial $kH$-module and $\Omega$ is the Heller operator. This easily follows from basic properties of the Heller operator (see Proposition~\ref{Heller prop}(v) with $V=\Omega^{-1} k$).
	
	A different example comes from the Clebsch--Gordan rule. 
	
	\begin{example}\label{endo examples}
		The $kG$-module $\Sym^{p-2} E$ is endotrivial since $\Sym^{p-2} E\otimes \Sym^{p-2} E\piso k$. Alternatively, we may write this as $\U_{p-2}^2 = \U_0$.   
	\end{example}
	
	A useful property of endotrivial modules is summarised by this well-known lemma.
	
	\begin{lemma}\label{endo lemma}
		Suppose that $V$ is an endotrivial $kH$-module and $W$ is an arbitrary non-projective indecomposable $kH$-module. Then the tensor product $V\otimes W$ has a unique non-projective indecomposable summand.
	\end{lemma}
	
	\begin{proof}
		We start by observing that the module $V\otimes W$ is not projective. This is true as otherwise the module $V^*\otimes V\otimes W$ would also be projective. But $V^*\otimes V\otimes W\piso k\otimes W \piso W$ using the endotriviality of $V$. Note that this observation can also be applied to the endotrivial module $V^*$ in place of $V$.
		
		Now, suppose that $V\otimes W$ has at least two non-projective indecomposable summands. By the previous paragraph so does $V^*\otimes\left( V\otimes W\right) $. But we have $V^*\otimes V\otimes W\piso W$, a contradiction.
	\end{proof}
	
	\begin{remark}\label{endo rmk}
		In fact, a more general statement that tensoring with an endotrivial $kH$-module $V$ defines a stable equivalence is true; see, for instance, \cite[p.~1]{CarlsonEndotrivialpGroups98}.
	\end{remark}
	
	We are now ready to prove the result about interchanging a Schur functor and tensoring with an endotrivial module. This follows rather easily once we establish that if $V$ is endotrivial, then for a $p$-small partition $\nu$ (recall this means that $|\nu| <p$) the module $\nabla^{\nu} V$ is almost always projective.
	
	\begin{proposition}\label{endo prop}
		Let $V$ be an endotrivial $kH$-module of dimension $d$ and let $\nu$ be any partition of size $n<p$.
		\begin{enumerate}[label=\textnormal{(\roman*)}]
			\item
			\begin{align*}
			V^{\otimes n}\piso	
			\begin{cases}
			\Sym^n V &\text{ if } d\equiv 1 \textnormal{ mod } p,\\
			\bigwedge^n V &\text{ if } d\equiv -1 \textnormal{ mod } p.\\ 
			\end{cases}
			\end{align*} 
			\item The module $\nabla^{\nu} V$ is projective unless $d\equiv 1$ mod $p$ and $\nu = (n)$ \emph{or} $d\equiv -1$ mod $p$ and $\nu = (1^n)$.  
			\item For $W$ an arbitrary $kH$-module
			\begin{align*}
			\nabla^{\nu}(V\otimes W)\piso
			\begin{cases}
			V^{\otimes n}\otimes \nabla^{\nu} W & \text{ if } d\equiv 1 \textnormal{ mod } p,\\
			V^{\otimes n}\otimes \nabla^{\nu'} W & \text{ if } d\equiv -1 \textnormal{ mod } p.\\
			\end{cases}
			\end{align*}
		\end{enumerate}
	\end{proposition}
	
	\begin{proof}
		Note that the module $V^{\otimes n}$ has precisely one non-projective indecomposable summand. This is because we can apply Lemma~\ref{endo lemma} $n$ times consecutively to $k$, $V$, $V^{\otimes 2}, \dots, V^{\otimes n-1}$, tensoring with $V$ each time. 
		\begin{enumerate}[label=\textnormal{(\roman*)}]
			\item Recall that the dimension of $\Sym^n V$ is $\binom{n+d-1}{n}$ and the dimension of $\bigwedge^n V$ is $\binom{d}{n}$. If $d\equiv 1$ mod $p$, then the first binomial coefficient is not divisible by $p$ (since the $n$ factors in the numerator are congruent to $n,n-1,\dots, 1$ mod $p$ and $n<p$), and hence $\Sym^n V$ is not projective. Similarly, if $d\equiv -1$ mod $p$, then the other binomial coefficient is not divisible by $p$ (this time the reminders are $-1,-2,\dots,-n$ and $n<p$) and $\bigwedge^n V$ is not projective. Since both $\Sym^n V$ and $\bigwedge^n V$ are summands of $V^{\otimes n}$ using Schur--Weyl duality (Theorem~\ref{SW duality thm}) and since the module $V^{\otimes n}$ has only one non-projective indecomposable summand, as noted above, this summand is (modulo projectives) $\Sym^n V$ if $d\equiv 1$ mod $p$ or $\bigwedge^n V$ if $d\equiv -1$ mod $p$.
			
			\item This follows from Schur--Weyl duality and (i).
			
			\item By Lemma~\ref{Schur functor expansion lemma}, the module $\nabla^{\nu}(V\otimes W)$ is isomorphic to the direct sum $\bigoplus_{\lambda, \mu \vdash n}g_{\lambda\mu}^{\nu} \left( \nabla^{\lambda} V \otimes \nabla^{\mu} W\right) $ where $g_{\lambda\mu}^{\nu}$ are the Kronecker coefficients. By (i) and (ii) we know that the only non-projective module $\nabla^{\lambda} V$ is either $\Sym^n V$ if $d\equiv 1$ mod $p$ or $\bigwedge^n V$ if $d\equiv -1$ mod $p$. Moreover, in both of these cases this module is isomorphic to $V^{\otimes n}$ modulo projectives. Hence we have
			\begin{align*}
			\nabla^{\nu}(V\otimes W)\piso
			\begin{cases}
			\bigoplus_{\mu\vdash n} g_{(n)\mu}^{\nu} \left( V^{\otimes n}\otimes \nabla^{\mu} W\right)  & \text{ if } d\equiv 1 \textnormal{ mod } p,\\
			\bigoplus_{\mu\vdash n} g_{(1^n) \mu}^{\nu}\left( V^{\otimes n}\otimes \nabla^{\mu} W\right)  & \text{ if } d\equiv -1 \textnormal{ mod } p.\\
			\end{cases}
			\end{align*}
			Recall from Lemma~\ref{Kronecker coefficients lemma} that $g_{(n)\mu}^{\nu}$ is $1$ if $\mu =\nu$ and zero otherwise and $g_{(1^n)\mu}^{\nu}$ is $1$ if $\mu =\nu'$ and zero otherwise. This immediately yields the result.
			\qedhere
		\end{enumerate}  
	\end{proof} 
	
	Applying Proposition~\ref{endo prop}(iii) to the endotrivial $kH$-module $\Omega k$ and the endotrivial $kG$-module $\Sym^{p-2} E$ from Example~\ref{endo examples} gives us two useful corollaries. In the proofs of these corollaries we freely use that if $\nu$ is a partition of $n$ with $n<p$, then the Schur functor $\nabla^{\nu}$ is well-defined in the stable module category (Corollary~\ref{well defined functors cor}).  
	
	\begin{corollary}\label{Heller endo cor}
		If $H$ is a finite group of order divisible by $p$, $W$ is a $kH$-module and $\nu$ is a partition of $n$ where $n<p$, then $\nabla^{\nu} (\Omega W)\piso \Omega^n (\nabla^{\nu'} W)$.
	\end{corollary}
	
	\begin{proof}
		Take $V=\Omega k$ in Proposition~\ref{endo prop}(iii). Since $k$ has dimension $1$ and its projective cover has dimension divisible $p$, we conclude that the dimension of $V$ is congruent to $-1$ modulo $p$. Hence $\nabla^{\nu}((\Omega k)\otimes W) \piso \left( \Omega k\right)^{\otimes n} \otimes \nabla^{\nu'} W$. This is the required result once tensoring with $\Omega k$ is replaced by applying the Heller operator $\Omega$ (according to Proposition~\ref{Heller prop}(v)).
	\end{proof}
	
	\begin{corollary}\label{p-2 endo cor}
		For $0\leq l\leq p-2$ and $\nu\vdash n$ with $n<p$ we have
		\begin{align*}
		\nabla^{\nu} \Sym^{p-2-l} E\piso
		\begin{cases}
		\nabla^{\nu'} \Sym^l E & \text{ if } n \text{ is even},\\
		\Sym^{p-2} E \otimes \nabla^{\nu'} \Sym^l E & \text{ if } n \text{ is odd}.\\
		\end{cases}
		\end{align*}
	\end{corollary}
	
	\begin{proof}
		Take $H=G$, $V=\Sym^{p-2} E$ and $W=\Sym^l E$ in Proposition~\ref{endo prop}(iii). The dimension of $V$ is $p-1$ which is $-1$ modulo $p$. Therefore we obtain $\nabla^{\nu}(\Sym^{p-2} E\otimes \Sym^l E)\piso \left( \Sym^{p-2} E \right) ^{\otimes n}\otimes \nabla^{\nu'} \Sym^l E$. We immediately get the desired result after using the rule for tensoring with $\Sym^{p-2} E$ modulo projectives from Corollary~\ref{CG rule cor}(ii).
	\end{proof}
	
	\section{The representation ring of $k\SL$ modulo projectives}
	
	Recall the rings $R_I$ and $R_E$ form \S\ref{notation sec}. We can describe these rings using a primitive $p$th root of unity, which we again denote by $\zeta_p$. This idea appears in Almkvist \cite[\S3]{AlmkvistReciprocity81} where the stable category of $k[\Z/p\Z]$ is analysed. One can recover the first two parts of the below result from Almkvist but we present an independent proof to make this paper self-contained. The statement we obtain is as follows.
	
	\begin{proposition}\label{Iso of R_I prop}
		Define a map $\Theta\!: R_I\to \Z[\zeta_p + \zeta_p^{-1}]$ of free abelian groups by $\Theta(\U_l) = \zeta_p^{-l} + \zeta_p^{-l +2} + \dots + \zeta_p^l$ for $0\leq l\leq p-2$. 
		\begin{enumerate}[label=\textnormal{(\roman*)}]
			\item The map $\Theta$ is a ring homomorphism.
			
			\item The map $\Theta$ is surjective and its kernel is generated by $\U_0+\U_{p-2}$ (as an ideal).
			
			\item The restriction $\Theta|_{R_E}$ is a ring isomorphism.
			
			\item The map $\Psi_0\!: R_E[Y]/(Y^2-1) \to R_I$ which maps $R_E$ into itself inside $R_I$ and $Y$ to $\U_{p-2}$ is an isomorphism of rings.
			
			\item The map $\Psi_1\!: \Z[\zeta_p + \zeta_p^{-1}][Y]/(Y^2-1) \to R_I$ given by identifying $\Z[\zeta_p + \zeta_p^{-1}]$ with $R_E$ using $(\Theta|_{R_E})^{-1}$ and by sending $Y$ to $\U_{p-2}$ is a ring isomorphism.    
		\end{enumerate}
	\end{proposition}

	\begin{remark}\label{l=p-1 rmk}
		The equality $\Theta(\U_l) = \zeta_p^{-l} + \zeta_p^{-l +2} + \dots + \zeta_p^l$ holds even for $l=p-1$ since for such $l$ it becomes $\Theta(0)=0$.
	\end{remark}
	
	\begin{proof}
		Throughout the proof if $R$ is a ring or an ideal, we refer to free generators of $R$ when considered as an abelian group as `free generators of $R$'. We similarly use the term `freely generated', meaning freely generated as an abelian group.
		We will use that $\Z[\zeta_p + \zeta_p^{-1}]$ is freely generated by the elements $(\zeta_p^2)^{-i} + (\zeta_p^2)^{-i+1} + \dots + (\zeta_p^2)^{i}$ with $0\leq i\leq (p-3)/2$. This is true as these elements are Galois conjugates of the free generators $\z^{-i} + \dots + \z^{i}$ (with $0\leq i\leq (p-3)/2$) of $\Z[\zeta_p + \zeta_p^{-1}]$.    
		\begin{enumerate}[label=\textnormal{(\roman*)}]
			\item Note that $\Theta(\U_0)=1$, which means the multiplicative identity is preserved by $\Theta$. Since $R_I$ is generated by $\U_1$ as a $\Z$-algebra (Corollary~\ref{CG rule cor}(iv)), to check that $\Theta$ is a ring homomorphism we just need to show that $\Theta(x\cdot\U_1) = \Theta(x)\Theta(\U_1)$ for all $x\in R_I$. In fact, we only need to check this for $x=\U_l$ for $0\leq l\leq p-2$ as these elements generate $R_I$ as an abelian group. This is a routine check using Corollary~\ref{CG rule cor}(i) and the equality $\zeta_p^{1-p} + \zeta_p^{3-p} + \dots + \zeta_p^{p-1}=0$.
			
			\item Since $\Theta(\U_0 + \U_{p-2}) = 1+\zeta_p^{2-p} + \dots + \zeta_p^{p-2} = \zeta_p^{-p} + \zeta_p^{2-p} + \dots + \zeta_p^{p-2} = \z^{-p}\left( 1+ \z^2 + (\z^2)^2 + \dots +(\z^2)^{p-1} \right)  = 0$, the element $\U_0 + \U_{p-2}$ lies in the kernel. Now, using Example~\ref{Ideal example}, the ideal generated by $\U_0 + \U_{p-2}$ is freely generated by $\U_{2i} + \U_{p-2-2i}$ with $0\leq i\leq (p-3)/2$. Thus the quotient ring $R_I/(\U_0 + \U_{p-2})$ is freely generated by the images of $\U_0, \U_2,\dots, \U_{p-3}$ and these elements are mapped by $\Theta$ to $(\zeta_p^2)^{-i} + (\zeta_p^2)^{-i+1} + \dots + (\zeta_p^2)^{i}$ with $0\leq i\leq (p-3)/2$, which are free generators of $\Z[\zeta_p + \zeta_p^{-1}]$. Hence $\Theta$ is surjective and the ideal $(\U_0 + \U_{p-2})$ is equal to the kernel of $\Theta$. 
			
			\item We have already mentioned that the free generators $\U_0, \U_2,\dots, \U_{p-3}$ of $R_E$ are mapped to free generators of $\Z[\zeta_p + \zeta_p^{-1}]$, and thus this part is immediate.
			
			\item By Corollary~\ref{CG rule cor}(ii) $\U_{p-2}^2=\U_0$, and hence $\Psi_0$ is a well-defined homomorphism. Moreover, by the same result, $R_I$ is freely generated by the elements $\U_{2i}$ and $\U_{2i}\cdot \U_{p-2}$ with $0\leq i\leq (p-3)/2$, which establishes that $\Psi_0$ is an isomorphism. 
			
			\item This is clear from (iv).
			\qedhere
		\end{enumerate}
	\end{proof}
	
	Recall from Example~\ref{p-lambda second example} and Example~\ref{p-lambda third example} that both $R_I$ and $\Z[\z + \z^{-1}]$ are $p$-$\lambda$-rings. The next result is not needed in our proof of Theorem~\ref{rep ring thm} but it is vital for establishing the main results. Since it concerns the map $\Theta$ from Proposition~\ref{Iso of R_I prop} we include it here. 
	
	\begin{proposition}\label{Iso of p-lambda-rings prop}
		The map $\Theta\!: R_I \to \Z[\z + \z^{-1}]$ introduced in Proposition~\ref{Iso of R_I prop} is a surjective homomorphism of $p$-$\lambda$-rings.
	\end{proposition}
	
	\begin{proof}
		We already know that $\Theta$ is a surjective homomorphism of rings. Hence it remains to establish that for all $x\in R_I$ and all $1\leq i\leq p-1$ we have $\Theta(\lambda^i x) = \lambda^i(\Theta(x))$. According to Corollary~\ref{CG rule cor}(iv), $R_I$ is generated by $\U_1$ as a $\Z$-algebra, and hence we only need to check that $\Theta(\lambda^i x) = \lambda^i(\Theta(x))$ for $x = \U_1$. But this is easy as both sides are zero if $i>2$ and for $i=1$ and $i=2$ both sides equal $\z + \z^{-1}$ and $1$, respectively. This establishes the result.
	\end{proof}   
	
	With Proposition~\ref{Iso of R_I prop}(v) established we are ready to prove Theorem~\ref{rep ring thm}.
	
	\begin{proof}[Proof of Theorem~\ref{rep ring thm}]
		Note that $\mathbb{Z}[\z +\z^{-1}][X,Y]/(X^{p-1}-1, Y^2-1)$ contains the ring $S:=\mathbb{Z}[\zeta_p + \zeta_p^{-1}][Y]/(Y^2-1)$ as a subring and that we can write $\mathbb{Z}[\z +\z^{-1}][X,Y]/(X^{p-1}-1, Y^2-1) \cong S[X]/(X^{p-1}-1)$. Recall the isomorphism $\Psi_1\!: S \to R_I$ from Proposition~\ref{Iso of R_I prop}(v). Since $\Z[\z +\z^{-1}] = \Z[\z^2 + \z^{-2}]$, $\Psi_1$ can be defined by $\Psi_1(\z^2 + \z^{-2}) = \U_2 -\U_0$ and $\Psi_1(Y) = \U_{p-2}$. Therefore we can see that $\Psi$ restricted to $S$ is just $\Psi_1$ (considered as a map to $\overline{R(G)}$ rather than to $R_I$), and thus it is an isomorphism onto $R_I\leq \overline{R(G)}$.
		
		Hence to show that $\Psi$ is a well-defined isomorphism we just need to show that the map $\Phi\!: R_I[X]/(X^{p-1}-1) \to \overline{R(G)}$ which maps $R_I$ identically to $R_I$ inside $\overline{R(G)}$ and $X$ to $\overline{\Omega k}$ is a well-defined isomorphism. But this is clear from Corollary~\ref{kG-Omega cor}(iii).  
	\end{proof}
	
	\begin{example}\label{height and position example}
		Let us arrange the non-projective indecomposable $kG$-modules into two tables with $p-1$ rows and $(p-1)/2$ columns as follows. Label the rows of both tables by \textit{heights} $h\in \Z/(p-1)\Z$ such that the bottom rows are labelled by $0$ and the label $h$ increases by $1$ as we move upwards. We label the columns by \textit{positions} $0\leq c\leq (p-3)/2$ from left to right. Now put $\Omega^h \left( \Sym^{2c} E \right) $ into row $h$ and column $c$ of the first table and $\Omega^h\left( \Sym^{p-2c-2} E \right) $ into row $h$ and column $c$ of the second table. See Table~\ref{height and position table} for the tables in the case $p=7$.
		
		\begin{table}[h!]
			\hspace*{-0.9cm}
			\centering
			\begin{tabular}{ccc}
				\centering
				\begin{tabular}{c}
					$h\backslash c$\\
					$5$\\
					$4$\\
					$3$\\
					$2$\\
					$1$\\
					$0$\\
				\end{tabular}&
				\centering
				\begin{tabular}{|ccc|}
					\hline
					$0$&$1$&$2$\\
					\hline
					$\Omega^5 k $&$\Omega^5\left( \Sym^2 E \right) $&$\Omega^5\left( \Sym^4 E \right) $\\
					$\Omega^4 k $&$\Omega^4\left( \Sym^2 E \right) $&$\Omega^4\left( \Sym^4 E \right) $\\
					$\Omega^3 k $&$\Omega^3\left( \Sym^2 E \right) $&$\Omega^3\left( \Sym^4 E \right) $\\
					$\Omega^2 k $&$\Omega^2\left( \Sym^2 E \right) $&$\Omega^2\left( \Sym^4 E \right) $\\
					$\Omega k $&$\Omega\left( \Sym^2 E \right) $&$\Omega\left( \Sym^4 E \right) $\\
					$k $&$\Sym^2 E $&$\Sym^4 E $\\
					\hline
				\end{tabular}&
				\begin{tabular}{|ccc|}
					\hline
					$0$&$1$&$2$\\
					\hline
					$\Omega^5\left( \Sym^5 E \right) $&$\Omega^5\left( \Sym^3 E \right) $&$\Omega^5 E $\\
					$\Omega^4\left( \Sym^5 E \right) $&$\Omega^4\left( \Sym^3 E \right) $&$\Omega^4 E $\\
					$\Omega^3\left( \Sym^5 E \right) $&$\Omega^3\left( \Sym^3 E \right) $&$\Omega^3 E $\\
					$\Omega^2\left( \Sym^5 E \right) $&$\Omega^2\left( \Sym^3 E \right) $&$\Omega^2 E $\\
					$\Omega\left( \Sym^5 E \right) $&$\Omega\left( \Sym^3 E \right) $&$\Omega E $\\
					$\Sym^5 E $&$\Sym^3 E $&$E $\\
					\hline
				\end{tabular}\\
			\end{tabular}
			\vspace{4pt}
			\caption{Two tables from Example~\ref{height and position example} with $p=7$ labelled by heights $h\in \Z/6\Z$ and positions $c\in\left\lbrace 0,1,2\right\rbrace $.}
			\label{height and position table}
		\end{table}
		
		By Corollary~\ref{kG-Omega cor}(ii) each non-projective indecomposable $kG$-module, up to isomorphism, is in precisely one box of our tables. It is worth mentioning that our two tables correspond to the two blocks of $kG$ with non-trivial defect groups (see \cite[Exercise~13.2]{AlperinLocal86}). Moreover, under the isomorphism $\Psi$ from Theorem~\ref{rep ring thm} the multiplication by $X$ corresponds to increasing the height by $1$ and the multiplication by $Y$ corresponds to moving to the corresponding box in the other table.
		
		More generally, when tensoring two non-projective indecomposable $kG$-modules one can treat the position, height and choice of table separately. In particular, if $V$ and $W$ are two non-projective indecomposable $kG$-modules with positions $c_V, c_W$, heights $h_V, h_W$, lying in tables $T_V, T_W$, respectively, then all the non-projective indecomposable summands of $V\otimes W$ have height $h_V + h_W$, lie in the first table if $T_V=T_W$ and in the second table otherwise and their positions are given by the Clebsch--Gordan rule (they agree with the positions of the non-projective indecomposable summands of $\Sym^{2c_V} E \otimes \Sym^{2c_W} E$). In the context of Theorem~\ref{rep ring thm} this just says that when multiplying monomials of $\mathbb{Z}[\z +\z^{-1}][X,Y]/(X^{p-1}-1, Y^2-1)$ we can add powers of $X$, powers of $Y$ and multiply elements of $\Z[\z  +\z^{-1}]$ separately.
		
		For instance, if $p=7$, the non-projective indecomposable summands of $\Omega^2(\Sym^2 E)\otimes \Omega(\Sym^3 E)$ have height $2+1=3$, they lie in the second table and their positions are $0,1$ and $2$. Thus $\Omega^2(\Sym^2 E)\otimes \Omega(\Sym^3 E)\piso \Omega^3(\Sym^5 E) \oplus \Omega^3(\Sym^3 E) \oplus \Omega^3 E$.
	\end{example}
	
	\section{Modular plethysms of $E$ and classifications}
	
	Recall that we call a module stably-irreducible if it has precisely one non-projective indecomposable summand which is moreover irreducible and we say that a partition $\nu$ is $p$-small if $|\nu|<p$.
	
	For $\nu$ a $p$-small partition and $0\leq l\leq p-2$ we know from Lemma~\ref{initial modular plethysms lemma}(i) that all the indecomposable summands of $\nabla^{\nu} \Sym^l E$ are either projective or irreducible. Hence we can apply the map $\Theta\!: R_I\to \Z[\zeta_p + \zeta_p^{-1}]$ from Proposition~\ref{Iso of R_I prop} to $\overline{\nabla^{\nu} \Sym^l E}$. While $\Theta$ is not injective, we can `invert' it in particular cases.
	
	\begin{lemma}\label{invertibility of Theta lemma}
		Let $\mu$ and $\nu$ be $p$-small partitions and $0\leq j,l\leq p-2$. If $\Theta(\overline{\nabla^{\mu} \Sym^j E})=\Theta(\overline{\nabla^{\nu} \Sym^l E})$, then $\nabla^{\mu} \Sym^j E\piso \nabla^{\nu} \Sym^l E$.
	\end{lemma}
	
	\begin{proof}
		The kernel of $\Theta$ is generated by $\U_{p-2} + \U_0$ as an ideal by Proposition~\ref{Iso of R_I prop}(ii). Thus it is freely generated as an abelian group by $\U_{2i} + \U_{p-2-2i}$ with $0\leq i \leq (p-3)/2$ (see Example~\ref{Ideal example}). Therefore $\overline{\nabla^{\mu} \Sym^j E} = \overline{\nabla^{\nu} \Sym^l E} + \sum_{i=0}^{(p-3)/2} a_i \left( \U_{2i} + \U_{p-2-2i}\right) $ for some integers $a_i$. We need to show that all the $a_i$ are zero.
		
		If some $a_i$ was positive, then $\Sym^{2i} E$ and $\Sym^{p-2-2i} E$ would both be summands of $\nabla^{\mu} \Sym^j E$; this contradicts Lemma~\ref{initial modular plethysms lemma}(ii) as $2i$ and $p-2-2i$ have different parities. Similarly, if some $a_i$ was negative, we would get that $\Sym^{2i} E$ and $\Sym^{p-2-2i} E$ are summands of $\nabla^{\nu} \Sym^l E$, a contradiction. Hence all the $a_i$ are zero, which establishes the result.
	\end{proof}
	
	We can now transform the notion of $\nabla^{\nu} \Sym^l E$ being stably-irreducible (or projective) to a property of $\Theta\left( \overline{\nabla^{\nu} \Sym^l E}\right) \in \Z[\zeta_p + \zeta_p^{-1}]$. To do this we introduce the following notation. For an integer $j$ write $g_j$ for $\z^{-j+1} + \z^{-j+3} + \dots + \z^{j-1} = \frac{\z^j - \z^{-j}}{\z - \z^{-1}}$. Thus $g_j\in \Z[\z + \z^{-1}]$ for all integers $j$, the sequence $g_j$ is periodic with period $p$ and from the definition of $\Theta$ we have the equality $g_j=\Theta(\U_{j-1})$ for $1\leq j\leq p-1$.  
	
	\begin{corollary}\label{Theta irreducibility cor}
		Let $\nu$ be a $p$-small partition and $0\leq l \leq p-2$.
		\begin{enumerate}[label=\textnormal{(\roman*)}]
			\item The modular plethysm $\nabla^{\nu} \Sym^l E$ is projective if and only if we have $\Theta(\overline{\nabla^{\nu} \Sym^l E})=0$.
			
			\item The modular plethysm $\nabla^{\nu} \Sym^l E$ is stably-irreducible if and only if $\Theta(\overline{\nabla^{\nu} \Sym^l E})$ equals $g_j$ for some $1\leq j\leq p-1$.
		\end{enumerate} 
	\end{corollary}
	\begin{proof}
		The `only if' directions are clear in both parts.
		\begin{enumerate}[label=\textnormal{(\roman*)}]
			\item The `if' direction comes from Lemma~\ref{invertibility of Theta lemma} applied to $p$-small partitions $(1^2)$ and $\nu$ and integers $0$ and $l$. Since $\bigwedge^2 k$ is the zero module we have $\Theta(\overline{\nabla^{\nu} \Sym^l E}) = 0 = \Theta(\overline{\bigwedge^2 k})$ and therefore $\nabla^{\nu} \Sym^l E \piso 0$.
			
			\item This time we use Lemma~\ref{invertibility of Theta lemma} applied to $p$-small partitions $(1)$ and $\nu$ and integers $j-1$ and $l$. Indeed, we have $\Theta(\overline{\nabla^{\nu} \Sym^l E}) = g_j = \Theta(\overline{\Sym^{j-1} E})$, and therefore $\nabla^{\nu} \Sym^l E \piso \Sym^{j-1} E$.
			\qedhere 
		\end{enumerate} 
	\end{proof}
	
	To use Corollary~\ref{Theta irreducibility cor} we need to understand $\Theta\left( {\overline{\nabla^{\nu} \Sym^l E}}\right)$. This is essentially established in \cite[Theorem~2.9.1]{BensonBundles17} in the setting of the cyclic group $\Z/p\Z$ rather than $G$ and without working modulo projectives. However, the same strategy using $p$-$\lambda$-rings (which are referred to as special $p$-$\lambda$-rings in \cite{BensonBundles17}) applies in our setting.
	
	\begin{theorem}\label{modular plethysm formula thm}
		Let $\nu$ be a $p$-small partition and $0\leq l\leq p-2$. Then
		\[
		\Theta\left( {\overline{\nabla^{\nu} \Sym^l E}}\right) = s_{\nu}(\z^{-l}, \z^{-l+2},\dots, \z^{l}),
		\]
		where $s_{\nu}$ is the Schur function labelled by the partition $\nu$. 
	\end{theorem}
	
	\begin{proof}
		By Proposition~\ref{Iso of p-lambda-rings prop} we know that $\Theta$ is a homomorphism of $p$-$\lambda$-rings. Since the operation $\left\lbrace \nu \right\rbrace $ is defined using only the $p$-$\lambda$-ring structure, we have for all $x\in R_I$ that $\Theta( \left\lbrace \nu\right\rbrace x) = \left\lbrace \nu\right\rbrace  \Theta(x)$. Taking $x=\U_l = \overline{\Sym^l E}$ we obtain the result according to the discussion at the end of \S\ref{p subsection}.
	\end{proof}
	
	\subsection{Projective classifications}
	We can immediately rewrite Theorem~\ref{modular plethysm formula thm} as follows.
	
	\begin{corollary}\label{SHCF cor}
		Let $\nu$ be a $p$-small partition and $0\leq l\leq p-2$. Then we can write
		\[
		\Theta\left( {\overline{\nabla^{\nu} \Sym^l E}}\right) = \frac{\prod_{c\in \mathcal{C}_{l+1}} (\z^c - \z^{-c})}{\prod_{h\in \mathcal{H}} (\z^h - \z^{-h})}.
		\]
		Using the introduced notation $g_j = \frac{\z^j - \z^{-j}}{\z - \z^{-1}}$ we can write
		\[
		\Theta\left( {\overline{\nabla^{\nu} \Sym^l E}}\right) = \frac{\prod_{c\in \mathcal{C}_{l+1}} g_c}{\prod_{h\in \mathcal{H}} g_h}.
		\]
	\end{corollary}
	
	\begin{proof}
		The first equality is a combination of Theorem~\ref{modular plethysm formula thm} and Theorem~\ref{SHCF thm} with a specialisation given by $q=\z$. The second equality follows from the first one after dividing the numerator and the denominator by $(\z+\z^{-1})^{|\nu|}$.
	\end{proof}
	
	To demonstrate the strength of Corollary~\ref{SHCF cor} we classify the projective modular plethysms, establishing our first main result Theorem~\ref{projective classification thm}.
	
	\begin{proof}[Proof of Theorem~\ref{projective classification thm}]
		The statement is trivially true if $\nu$ is the empty partition. Also, the statement holds if $\ell(\nu)\geq l+2$ as then $\nabla^{\nu} \Sym^l E=0$ since the dimension of $\Sym^l E$ is $l+1$. Now suppose that $\nu\neq \o$ and $\ell(\nu)\leq l+1$.
		
		Combining Corollary~\ref{Theta irreducibility cor}(i) and Corollary~\ref{SHCF cor}, the modular plethysm $\nabla^{\nu} \Sym^l E$ is projective if and only if
		\begin{equation}\label{projective equation}
		\frac{\prod_{c\in \mathcal{C}_{l+1}} g_c}{\prod_{h\in \mathcal{H}} g_h} =0.
		\end{equation}
		Recall that $g_j=\frac{\z^j - \z^{-j}}{\z-\z^{-1}}$, and thus $g_j=0$ if and only if $p$ divides $j$. Since $\nu$ is $p$-small, all the elements of $\mathcal{H}$ lie between $1$ and $p-1$, and hence the denominator in (\ref{projective equation}) is non-zero.
		
		The numerator in (\ref{projective equation}) is zero precisely when there is $c\in \mathcal{C}_{l+1}$ such that $p\mid c$. Using Lemma~\ref{shifted content lemma}(iii), this is equivalent to the existence of $c$ divisible by $p$ in the range $l+2-\ell(\nu)\leq c\leq \nu_1 +l$. Observe that $1\leq l+2-\ell(\nu)\leq p$ by our assumptions. Thus if there is such $c$, it can be chosen to be $p$. And $p$ lies in this interval if and only if $p\leq \nu_1 +l$, which finishes the proof.    
	\end{proof}
	
	We can immediately prove the third main result as well.
	
	\begin{proof}[Proof of Theorem~\ref{final projective classification thm}]
		Using Corollary~\ref{kG-Omega cor}(ii) we can write any indecomposable $kG$-module, up to isomorphism, as $\Omega^i (\Sym^l E)$ for some $0\leq i\leq p-2$ and $0\leq l\leq p-2$. Applying Corollary~\ref{Heller endo cor} $i$ times we can write $\nabla^{\nu}\left( \Omega^i (\Sym^l E) \right) \piso \Omega^{i|\nu|} \left( \nabla^{\lambda} \Sym^l E \right) $ where $\lambda$ is defined as in the statement of the theorem. Hence, using Proposition~\ref{Heller prop}(i), $\nabla^{\nu}\left( \Omega^i (\Sym^l E) \right)$ is projective if and only if $\nabla^{\lambda} \Sym^l E$ is projective. The result therefore follows from Theorem~\ref{projective classification thm}.
	\end{proof}
	
	To obtain the classification of the stably-irreducible modular plethysms we proceed similarly as in the projective case. In particular, we change the problem into a question about the multisets $\mathcal{C}_{l+1}$ and $\mathcal{H}$. However, before doing so we introduce a few reduction results, which simplify the argument.
	
	\begin{lemma}\label{reduction lemma}
		Let $\nu$ be a $p$-small partition and $0\leq l\leq p-2$.
		\begin{enumerate}[label=\textnormal{(\roman*)}]
			\item If $\ell(\nu) = l+1$ and $\mu$ is a partition obtained from $\nu$ by removing the first column (that is decreasing the first $l+1$ parts of $\nu$ by $1$), then $\nabla^{\nu} \Sym^l E \piso \nabla^{\mu} \Sym^l E$.
			
			\item If $\nu_1 = p-l-1$ and $\mu$ is a partition obtained from $\nu$ by removing the first row (that is $\mu = (\nu_2, \nu_3,\dots)$), then either $\nabla^{\nu} \Sym^l E \piso \nabla^{\mu} \Sym^l E$ or $\nabla^{\nu} \Sym^l E \piso \Sym^{p-2} E \otimes \nabla^{\mu} \Sym^l E$.
			
			\item In the setting of (ii) the modular plethysm $\nabla^{\nu} \Sym^l E$ is stably-irreducible if and only if $\nabla^{\mu} \Sym^l E$ is stably-irreducible.  
		\end{enumerate}
	\end{lemma}
	
	In our proof of Lemma~\ref{reduction lemma} we use the following combinatorial lemma. It is a straightforward application of the combinatorial definition of Schur functions \cite[Definition~7.10.1]{StanleyEnumerativeII99}. 
	
	\begin{lemma}\label{combinatorial lemma}
		Let $\lambda$ be a partition with $m=\ell(\lambda)$. If $\mu$ is a partition obtained from $\lambda$ by removing the first column, then $s_{\lambda}(\List{x}{m}) = x_1 x_2\dots x_m s_{\mu}(\List{x}{m})$.
	\end{lemma}
	
	\begin{proof}
		The combinatorial definition of Schur functions is $s_{\lambda}(\List{x}{m}) = \sum_{T} x^T$, where $T$ runs over the set $S(\lambda)_{m}$ of standard tableaux of shape $\lambda$ with entries from $\left\lbrace 1,2,\dots, m\right\rbrace $.
		
		Since $\ell(\lambda)=m$, any standard tableau $T\in S(\lambda)_m$ has the first column filled with $1,2,\dots, m$ from top to bottom. Therefore there is a bijection $r$ between $S(\lambda)_{m}$ and $S(\mu)_{m}$ given by removing the first column with an inverse given by adding a new first column filled with $1,2,\dots,m$ from top to bottom. Moreover, $x^T = x_1x_2\dots x_{m} x^{r(T)}$ for all $T\in S(\lambda)_m$ which in turn gives the required equality $s_{\lambda}(\List{x}{m}) = x_1x_2\dots x_{m} s_{\mu}(\List{x}{m})$.
	\end{proof}
	
	\begin{proof}[Proof of Lemma~\ref{reduction lemma}]\leavevmode
		\begin{enumerate}[label=\textnormal{(\roman*)}]
			\item Using Lemma~\ref{invertibility of Theta lemma}, we need to show $\Theta(\overline{\nabla^{\nu} \Sym^l E}) = \Theta(\overline{\nabla^{\mu} \Sym^l E})$. Using Theorem~\ref{modular plethysm formula thm} this equality becomes $s_{\nu}(\z^{-l}, \z^{-l+2}, \dots, \z^{l}) = s_{\mu}(\z^{-l}, \z^{-l+2}, \dots, \z^{l})$.
			
			Applying Lemma~\ref{combinatorial lemma} with $\lambda = \nu$ and $m=l+1$ we get that $s_{\nu}(x_1, x_2, \dots, x_{l+1})$ and $x_1 x_2\dots x_{l+1} s_{\mu}(x_1, x_2, \dots, x_{l+1})$ equal. But this becomes the desired equality after specialising $\List{x}{l+1}$ to $\z^{-l},  \z^{-l+2}, \dots, \z^{l}$ since their product then equals $1$.
			
			\item Note that $\nu'$ and $\mu'$ satisfy the constraints in (i) with $l$ replaced by $p-l-2$, thus $\nabla^{\nu'} \Sym^{p-l-2} E \piso \nabla^{\mu'} \Sym^{p-2-l} E$. Now, applying Corollary~\ref{p-2 endo cor} to both sides, we obtain $(\Sym^{p-2} E \otimes) \nabla^{\nu} \Sym^l E \piso (\Sym^{p-2} E \otimes) \nabla^{\mu} \Sym^l E$, where $\Sym^{p-2} E \otimes$ in brackets denotes that $\Sym^{p-2} E \otimes$ may or may not appear (depending on the parities of $|\mu|$ and $|\nu|$). Now, we are done since $\Sym^{p-2} E \otimes \Sym^{p-2} E \piso k$, and hence the module $\Sym^{p-2} E$ can be moved to the right-hand side if necessary.
			
			\item By Lemma~\ref{endo lemma} and Example~\ref{endo examples}, tensoring with $\Sym^{p-2} E$ preserves the number of non-projective indecomposable summands. Thus by (ii), the modular plethysm $\nabla^{\nu} \Sym^l E$ has precisely one non-projective indecomposable summand if and only if $\nabla^{\mu} \Sym^l E$ does. Hence, using Lemma~\ref{initial modular plethysms lemma}(i), one is stably-irreducible if and only if the other is. 	
			\qedhere    
		\end{enumerate}
	\end{proof}

\begin{remark}\label{involution remark}
	An argument similar to the one given in (iii) combined with Corollary~\ref{p-2 endo cor} shows that for any $p$-small partition $\nu$ and $0\leq l\leq p-2$ the modular plethysm $\nabla^{\nu} \Sym^l E$ is stably-irreducible (or projective) if and only if $\nabla^{\nu'} \Sym^{p-l-2} E$ is stably-irreducible (or projective). This agrees with Theorem~\ref{projective classification thm} and Theorem~\ref{full irred modular plethysms of E thm}.
\end{remark}
	
	Observe that we do not need to consider all $p$-small partitions in our search for all the stably-irreducible modular plethysms. Indeed, according to Theorem~\ref{projective classification thm}, we can rule out the partitions with more than $l+1$ parts and the partitions with the first part of size larger than $p-l-1$. Moreover, according to Lemma~\ref{reduction lemma}(i)~and~(iii), we can also ignore the cases when $l(\nu) =l+1$ or $\nu_1 = p-l-1$. This motivates the following definition.
	
	\begin{definition}
		We say that a partition $\nu$ is \textit{$(p,l)$-small} if the following conditions hold:
		\begin{enumerate}[label=\textnormal{(\roman*)}]
			\item $\nu$ is $p$-small,
			\item $\nu_1 \leq p-l-2$,
			\item $\ell(\nu)\leq l$.
		\end{enumerate}
	\end{definition}
	
	\subsection{Cyclotomic units and multisets}
	For $1\leq j\leq p-1$ the element $g_j = \frac{\z^j - \z^{-j}}{\z - \z^{-1}} = \Theta(\U_{j-1})$ is a cyclotomic unit in $\Z[\z + \z^{-1}]$ (see \cite[\S8.1]{WashingtonCyclotomic97}). By Proposition~\ref{Iso of R_I prop}(ii), the element $\U_{j-1}\cdot(\U_0 + \U_{p-2}) = \U_{j-1} + \U_{p-j-1}$ lies in the kernel of $\Theta$. Hence we obtain the identities $g_j = -g_{p-j}$ for $1\leq j\leq p-1$. Let $G_C$ be the subgroup of the group of cyclotomic units generated by $g_j$ with $1\leq j\leq p-1$. From the preceding discussion we can already conclude that $G_C$ is generated by $-1=g_{p-1}$ and $g_j$ with $2\leq j \leq (p-1)/2$. In fact, a stronger statement fully describing $G_C$ is true.
	
	\begin{theorem}\label{alg independence thm}
		The torsion-free part of $G_C$ is freely generated by $g_j$ with $2\leq j\leq (p-1)/2$ and the torsion part is a group of order two generated by $-1$.
	\end{theorem}
	
	\begin{proof}
		See \cite[Theorem~8.3]{WashingtonCyclotomic97} applied to $n=p$ and the $p$th primitive root of unity $\z^2$.
	\end{proof}
	
	Let us now establish our terminology for working with multisets. Fix a finite multiset $\mathcal{M}$ with integral elements and an integer $x$. We write $m_{\mathcal{M}}(x)$ for the multiplicity of $x$ in $\mathcal{M}$ and $|\mathcal{M}|$ for the size of $\mathcal{M}$, that is the sum of multiplicities of all elements in $\mathcal{M}$.
	
	We use the notation $\mathcal{M}\cup \left\lbrace  x\right\rbrace  $ for the multiset where each integer but $x$ has the same multiplicity as in $\mathcal{M}$ and the multiplicity of $x$ is $m_{\mathcal{M}}(x)+1$. In other words, $\mathcal{M}\cup \left\lbrace  x\right\rbrace $ is obtained from $\mathcal{M}$ by adding $x$.
	
	In the case that $x$ lies in $\mathcal{M}$, we denote by $\mathcal{M}\setminus \left\lbrace  x\right\rbrace  $ the multiset obtained from $\mathcal{M}$ by removing $x$ (thus the multiplicity of $x$ is $m_{\mathcal{M}}(x)-1$).
	
	Finally, we use $\mathcal{M}\setminus \left\lbrace  x\right\rbrace ^* $ for the multiset obtained from $\mathcal{M}$ by removing all occurrences of $x$ (thus $m_{\mathcal{M}\setminus \left\lbrace  x\right\rbrace ^*}(x)=0$). This multiset has size equal to $|\mathcal{M}|-m_{\mathcal{M}}(x)$. 
	
	We introduce the following notation to take care of the relations $g_j = -g_{p-j}$ in $G_C$.
	
	\begin{definition}
		Let $\mathcal{M}$ be a multiset with elements from $\left\lbrace 1,\dots, p-1\right\rbrace $. We denote by $\mathcal{M}^F$ the multiset obtained from $\mathcal{M}$ by replacing all occurrences of $i$ by $p-i$ for all $i\geq (p+1)/2$. Hence all the elements of $\mathcal{M}^F$ lie between $1$ and $(p-1)/2$ and for each $i$ in this range $m_{\mathcal{M}^F}(i) = m_{\mathcal{M}}(i) + m_{\mathcal{M}}(p-i)$. We refer to $\mathcal{M}^F$ as the \textit{fold of $\mathcal{M}$}. 
	\end{definition}
	
	The restriction to $(p,l)$-small partitions allows us to establish the following result.
	
	\begin{lemma}\label{C does not have 1 lemma}
		Let $0\leq l\leq p-2$ and let $\nu$ be a non-empty $(p,l)$-small partition.
		\begin{enumerate}[label=\textnormal{(\roman*)}]
			\item The multisets $\mathcal{H}^F$ and $\mathcal{C}_{l+1}^F$ are well-defined (meaning that $\mathcal{H}$ and $\mathcal{C}_{l+1}$ contain elements between $1$ and $p-1$ only).
			\item $\mathcal{C}_{l+1}^F$ does not contain $1$.
			\item Suppose that $p\neq 3$. Let $t\in\left\lbrace 0,1,2 \right\rbrace $ denote the number of equalities among $\ell(\nu) = l$ and $\nu_1 = p-l-2$ which hold. Then $m_{\mathcal{C}_{l+1}^F}(2)=t$.  
		\end{enumerate}
	\end{lemma}
	
	\begin{proof}
		As $\nu$ is $p$-small, the hook lengths of $\nu$ lie between $1$ and $p-1$, establishing the statement about $\mathcal{H}$. By Lemma~\ref{shifted content lemma}(i)~and~(ii), the maximal element of $\mathcal{C}_{l+1}$ is $\nu_1+l\leq p-2$ and the minimal element of $\mathcal{C}_{l+1}$ is $l+2-\ell(\nu)\geq 2$ and both of them have multiplicity one in $\mathcal{C}_{l+1}$. (i) and (ii) follow immediately.
		
		For $p\geq 5$ we have $2< p-2$ and $m_{\mathcal{C}_{l+1}^F}(2) = m_{\mathcal{C}_{l+1}}(2) + m_{\mathcal{C}_{l+1}}(p-2)$. From the above inequalities regarding the maximal and the minimal element of $\mathcal{C}_{l+1}$ we obtain $m_{\mathcal{C}_{l+1}}(2)\leq 1$ with equality if and only if $\ell(\nu) =l$ and $m_{\mathcal{C}_{l+1}}(p-2)\leq 1$ with equality if and only if $\nu_1 = p-l-2$. This establishes (iii).
	\end{proof}
	
	\begin{remark}\label{p=3 remark}
		Even if $p=3$, (iii) remains true. This is because for $p=3$ and $0\leq l\leq 1$ there is no non-empty $(p,l)$-small partition.

		One can also allow $\nu=\o$ in (i) and (ii) since then the involved multisets are empty. 
	\end{remark}
	
	We are ready to prove the following key result. Note that if $\mathcal{M}$ and $\mathcal{N}$ are two finite multisets of integers, we interpret the statement $\mathcal{M} = \mathcal{N}\setminus\left\lbrace  1\right\rbrace  $ as $1$ lies in $\mathcal{N}$ and $\mathcal{M}$ equals $\mathcal{N}\setminus\left\lbrace  1\right\rbrace $.
	
	\begin{proposition}\label{multiset prop}
		Let $0\leq l\leq p-2$ and let $\nu$ be a $(p,l)$-small partition. Then $\nabla^{\nu} \Sym^l E$ is stably-irreducible if and only if there is an integer $i$ such that $\mathcal{C}_{l+1}^F = \left( \mathcal{H}^F\cup \left\lbrace  i\right\rbrace \right)\setminus \left\lbrace  1\right\rbrace  $.
	\end{proposition} 
	
	\begin{proof}
		Using Corollary~\ref{Theta irreducibility cor}, the modular plethysm $\nabla^{\nu} \Sym^l E$ is stably-irreducible if and only if $\Theta(\overline{\nabla^{\nu} \Sym^l E})$ is equal to $g_j$ for some $1\leq j\leq p-1$. Rewrite $\Theta(\overline{\nabla^{\nu} \Sym^l E})$ using Corollary~\ref{SHCF cor} to get an equivalent condition: there exists $1\leq j\leq p-1$ such that
		\[
		\frac{\prod_{c\in \mathcal{C}_{l+1}} g_c}{\prod_{h\in \mathcal{H}} g_h} =  g_j.
		\] 
		Now, all the indices involved lie between $1$ and $p-1$ (using Lemma~\ref{C does not have 1 lemma}(i) and Remark~\ref{p=3 remark}), and thus we can replace each $g_m$ with $m\geq (p+1)/2$ by $-g_{p-m}$. In other words, we can replace the multisets $\mathcal{C}_{l+1}$ and $\mathcal{H}$ by their folds, $j$ by $p-j$ if necessary and potentially add a minus sign. Thus we get the following valid statement: $\nabla^{\nu} \Sym^l E$ is stably-irreducible if and only if there are $1\leq i\leq (p-1)/2$ and $\varepsilon\in \left\lbrace \pm 1\right\rbrace $ such that
		\begin{equation}\label{SCHF equation}
		\frac{\prod_{c\in \mathcal{C}_{l+1}^F} g_c}{\prod_{h\in \mathcal{H}^F} g_h} = \varepsilon g_i.
		\end{equation}
		This can be also written as $\prod_{c\in \mathcal{C}_{l+1}^F} g_c = \varepsilon \prod_{h\in \mathcal{H}^F\cup \left\lbrace  i\right\rbrace  } g_h$. Each $g_j$ involved is either $1$ (if $j=1$) or is labelled by $j$ in the range $2,3,\dots, (p-1)/2$. Therefore, using Theorem~\ref{alg independence thm}, we see that (\ref{SCHF equation}) holds if and only if $\varepsilon =1$ and $\mathcal{C}_{l+1}^F\setminus\left\lbrace  1\right\rbrace  ^* =\left( \mathcal{H}^F\cup\left\lbrace  i\right\rbrace \right) \setminus\left\lbrace  1\right\rbrace  ^*$.
		
		Thus $\nabla^{\nu} \Sym^l E$ is stably-irreducible if and only if there exists $1\leq i\leq (p-1)/2$ such that $\mathcal{C}_{l+1}^F\setminus\left\lbrace  1\right\rbrace  ^* =\left( \mathcal{H}^F\cup\left\lbrace  i\right\rbrace \right) \setminus\left\lbrace  1\right\rbrace  ^*$. Moreover, we can omit the range of $i$ in this statement as it comes implicitly from the equality of the multisets.
		
		To finish, notice that for two finite multisets of integers $\mathcal{M}$ and $\mathcal{N}$ with $|\mathcal{N}| = |\mathcal{M}| + 1$ the statements $\mathcal{M}\setminus\left\lbrace  1\right\rbrace  ^* = \mathcal{N}\setminus\left\lbrace  1\right\rbrace  ^*$ and $\mathcal{M} = \mathcal{N} \setminus \left\lbrace  1\right\rbrace  $ are equivalent. Applying this with $\mathcal{M} = \mathcal{C}_{l+1}^F$ and $\mathcal{N}=\mathcal{H}^F\cup \left\lbrace  i\right\rbrace $ yields the result.   
	\end{proof}
	
	\subsection{Stably-irreducible classifications}
	
	We deduce Theorem~\ref{full irred modular plethysms of E thm} from the classification of the stably-irreducible modular plethysms $\nabla^{\nu}\Sym^l E$ with $0\leq l\leq p-2$ and $\nu$ a $(p,l)$-small partition.
	
	\begin{theorem}\label{irred modular plethysms of E thm}
		Let $0\leq l\leq p-2$ and let $\nu$ be a $(p,l)$-small partition. Then $\nabla^{\nu} \Sym^l E$ is stably-irreducible if and only if (at least) one of the following happens:
		\begin{enumerate}[label=\textnormal{(\roman*)}]
			\item \textnormal{(elementary cases)} $\nu = \o$ or $\nu=(1)$,
			\item \textnormal{(row cases)} $\nu=(p-l-2)$ or $l=1$,
			\item \textnormal{(column cases)} $\nu=(1^l)$ or $l=p-3$,
			\item \textnormal{(rectangular cases)} $p=7$ and either $\nu=(2,2,2)$ with $l=3$ or $\nu=(3,3)$ with $l=2$.
		\end{enumerate} 
	\end{theorem}
	
	\begin{remark}\label{cases rmk}
		Note that in (ii) $\nu$ is a row (or the empty) partition. This is obvious if $\nu=(p-l-2)$ and it is implicit if $l=1$ since $\nu$ is $(p,l)$-small and thus $\ell(\nu)\leq l=1$. Hence the label `row cases'. An analogous statement can be made for (iii).
	\end{remark}
	
	\begin{proof}
		Clearly $\nabla^{\nu} \Sym^l E$ is stably-irreducible if $\nu$ is the empty partition. The statement is also trivially true if $p=3$ using Remark~\ref{p=3 remark}. Let us now suppose that $\nu\neq \o$ and $p\geq 5$ (and thus $2\leq (p-1)/2$).
		
		By Proposition~\ref{multiset prop} we can classify pairs $(\nu,l)$ for which there is an integer $i$ such that $\mathcal{C}_{l+1}^F = \left( \mathcal{H}^F\cup\left\lbrace  i\right\rbrace \right) \setminus\left\lbrace  1\right\rbrace   $. From Lemma~\ref{C does not have 1 lemma}(ii), $\mathcal{C}_{l+1}^F$ does not contain $1$ and hence we need $m_{\mathcal{H}^F}(1)\leq 1$. This means that $\nu$ has exactly one removable box. Therefore $\nu$ is a rectangular partition $(a^b)$ for some $a,b\in \mathbb{N}$.
		
		We now consider four cases given by distinguishing $a=1$ from $a\neq 1$ and $b=1$ from $b\neq 1$. The inequality
		
		\begin{equation}\label{2 inequality}
		m_{\mathcal{C}_{l+1}^F}(2)\geq  m_{\mathcal{H}^F}(2)
		\end{equation}
		obtained from $\mathcal{C}_{l+1}^F = \left( \mathcal{H}^F\cup\left\lbrace  i\right\rbrace \right) \setminus\left\lbrace  1\right\rbrace  $ is used throughout.
		
		Note the labelling below corresponds to the labelling used in the statement of the theorem. In each but the first case we firstly deduce necessary constraints on $\nu$ and $l$ and then check that they give rise to stably-irreducible modular plethysms.
		
		\begin{enumerate}[label=\textnormal{(\roman*)}]
			\item If $a=b=1$, we have $\nu=(1)$ and $\nabla^{\nu} \Sym^l E = \Sym^l E$ which is a non-projective irreducible module.
			
			\item If $a>1$ and $b=1$, then $\mathcal{H}^F$ contains $2$ and hence $m_{\mathcal{C}_{l+1}^F}(2)\geq 1$ by (\ref{2 inequality}). Using Lemma~\ref{C does not have 1 lemma}(iii), either $a=p-2-l$, that is $\nu = (p-2-l)$, or $b=l$, that is $l=1$.
			
			Let us check that in both cases we actually get stably-irreducible modular plethysms. For $l=1$ the partition $\nu$ has implicitly at most one row (see Remark~\ref{cases rmk}). Thus $\nu = (a)$ for some $0\leq a \leq p-3$ (the upper bound comes from the definition of $(p,l)$-small partitions). Hence $\nabla^{\nu} \Sym^l E = \Sym^a E$ is a non-projective irreducible module.
			
			In the former case $\nu = (p-2-l)$ we get $\mathcal{H} = \left\lbrace 1,2,\dots, p-2-l\right\rbrace $ and $\mathcal{C}_{l+1} = \left\lbrace l+1, l+2, \dots, p-2 \right\rbrace$. Thus $\mathcal{C}_{l+1}^F = \left( \mathcal{H}^F\cup\left\lbrace  i\right\rbrace \right) \setminus\left\lbrace  1\right\rbrace $ with $i=l+1$ or $i=p-l-1$.
			
			\item If $a=1$ and $b>1$, then we proceed as in (ii). One obtains two families of stably-irreducible modular plethysms given by $\nu = (1^l)$ and $l=p-3$.
			
			\item If $a,b\geq 2$, then $m_{\mathcal{H}^F}(2)\geq 2 $. Using (\ref{2 inequality}) we have $m_{\mathcal{C}_{l+1}^F}(2)\geq 2$, which implies that $a=p-2-l$ and $b=l$ by Lemma~\ref{C does not have 1 lemma}(iii). Since $\nu$ is $p$-small we also need $ab = |\nu| <p$. This becomes $pl-2l-l^2\leq p-1$ which can be rearranged as $(p-l-3)(l-1)\leq 2$. This inequality is easy to solve as both brackets on the left-hand side are positive integers (they equal $a-1$ and $b-1$, respectively). We get $p=7$ together with $\nu = (2,2,2)$ and $l=3$ \emph{or} $\nu = (3,3)$ and $l=2$.
			
			Both corresponding modular plethysms are stably-irreducible as $\mathcal{C}_{l+1}^F = \left( \mathcal{H}^F\cup\left\lbrace  3\right\rbrace \right) \setminus\left\lbrace  1\right\rbrace $, according to Figure~\ref{folded H and C figure}.
			\qedhere  
		\end{enumerate}
		
		\begin{figure}[h]
			\Yboxdim19pt
			$\young(33,32,21) \quad  \young(32,33,23) \qquad \young(332,321) \quad \young(332,233)$
			\caption{The Young diagrams of partitions $\nu = (2,2,2)$ and $\nu=(3,3)$ displaying $\mathcal{H}^F$ and $\mathcal{C}_{l+1}^F$ from left to right with $l=3$ and $l=2$, respectively.}
			\label{folded H and C figure}
		\end{figure}  
	\end{proof}
	
	In light of Lemma~\ref{reduction lemma} we can now add columns of length $l+1$ and rows of length $p-l-1$ to prove the second main result Theorem~\ref{full irred modular plethysms of E thm}, extending this classification to all $p$-small partitions. Note that we can never add both a column of size $l+1$ and a row of length $p-l-1$ as then our new partition would not be $p$-small.
	
	\begin{proof}[Proof of Theorem~\ref{full irred modular plethysms of E thm}]
		We can see that adding $b$ rows of length $p-l-1$ to partitions in the elementary cases of Theorem~\ref{irred modular plethysms of E thm} yields precisely the elementary cases augmented by rows. Similarly for adding $a$ columns of length $l+1$, we recover the elementary cases augmented by columns.
		
		Now adding $b$ rows of length $p-l-1$ to the partition $\nu=(p-l-2)$ from the row cases yields the first family in the augmented row cases, while when it comes to columns, we can add just one (to end up with a $p$-small partition) which gives us the hook case. When we move to the case $l=1$ our original $(p,l)$-small partition can be any row (or the empty) partition $(a)$ with $0\leq a\leq p-3$. We can now add rows of length $p-2$ or columns of length $2$. Therefore we can obtain any $p$-small partition contained inside the box $2\times (p-2)$. Conversely, for any $p$-small partition $\nu$ not contained inside the box $2\times (p-2)$ the $kG$-module $\nabla^{\nu} \Sym^1 E = \nabla^{\nu} E$ is projective (and thus not stably-irreducible) by Theorem~\ref{projective classification thm}. Hence we obtain the rest of the augmented row cases.
		
		The column cases are analogous to the row cases. We obtain the augmented column cases and again the hook case.
		
		Finally, we cannot add any more boxes to the rectangular cases as we already have $6=p-1$ boxes and we need our partition to be $p$-small. This gives us the final case in Theorem~\ref{full irred modular plethysms of E thm}, finishing the proof.   
	\end{proof} 
	
	We are ready to prove Theorem~\ref{final classification thm}. Recall that we use the term stably-irreducible pairs for pairs $(\nu, l)$ of a $p$-small partition $\nu$ and $0\leq l \leq p-2$ such that $\nabla^{\nu} \Sym^l E$ is stably-irreducible. The stably-irreducible pairs are therefore classified by Theorem~\ref{full irred modular plethysms of E thm}.
	
	\begin{proof}[Proof of Theorem~\ref{final classification thm}]
		We know from Corollary~\ref{kG-Omega cor}(ii) that every non-projective indecomposable $kG$-module can be written, up to isomorphism, as $\Omega^i (\Sym^l E)$ for some $0\leq i\leq p-2$ and $0\leq l\leq p-2$. Using Corollary~\ref{Heller endo cor} $i$ times yields $\nabla^{\nu}\left(\Omega^i \left( \Sym^l E\right)  \right) \piso \Omega^{i|\nu|}\left( \nabla^{\lambda} \Sym^l E \right) $ where $\lambda$ is defined as in the statement of the theorem.
		
		By Lemma~\ref{initial modular plethysms lemma}(i), the non-projective indecomposable summands of the modular plethysm $\nabla^{\lambda} \Sym^l E$ are irreducible. Hence, using Corollary~\ref{kG-Omega cor}(i) and Proposition~\ref{Heller prop}(i) and (iii), $p-1$ has to divide $i|\nu|$ so $\Omega^{i|\nu|}\left( \nabla^{\lambda} \Sym^l E \right)$ has a chance to have some non-projective irreducible summands. But in such a case $\Omega^{i|\nu|}\left( \nabla^{\lambda} \Sym^l E \right) $ is just $\nabla^{\lambda} \Sym^l E$ which is, from the definition, stably-irreducible if and only if $(\lambda, l)$ is a stably-irreducible pair.
	\end{proof}

	\subsection*{Acknowledgements} The author would like to thank Mark Wildon for suggesting this project and providing Magma code used for experimental calculations at the initial stage of the project, and an anonymous referee for helpful suggestions and corrections.
	
	\bibliographystyle{alpha}
	\bibliography{MSNrefs}
\end{document}